\documentclass[11pt,letterpaper]{article}

\usepackage{epsfig}
\usepackage[margin=1in]{geometry}
\usepackage{amssymb, amsmath, amsthm, graphicx, subcaption, mathtools, tikz-cd}
\usepackage[section]{algorithm}
\usepackage{enumitem}
\usepackage{nicematrix, tikz}
\usepackage[dvipsnames]{xcolor}
\usepackage{hyperref}
\usepackage{mathabx}
\usepackage{booktabs}
\usepackage{siunitx}
\hypersetup{
  colorlinks = true,
  linkcolor  = RoyalBlue,
  citecolor  = ForestGreen,
  urlcolor   = Magenta
}

\usepackage[backend=biber,
            style=numeric,
            sorting=none,
            maxnames=3]{biblatex}
\newcommand{\E}{\mathbb{E}}
\newcommand{\Var}{\mathrm{Var}}
\newcommand{\Prob}{\mathbb{P}}
\newcommand{\Cov}{\mathrm{Cov}}

\newcommand{\Yione}{Y_i(1)}
\newcommand{\Yizero}{Y_i(0)}
\newcommand{\T}{\text{T}}
\newcommand{\taui}{\tau_i}
\newcommand{\taubar}{\widehat{\widebar{\tau}}_{\text{avg}}}
\newcommand{\ate}{\bar{\tau}}
\newcommand{\sigmafield}{\mathcal{F}_{i-1}}
\newcommand{\snipw}{\widehat{\tau}_{\textsc{IPW}}}
\newcommand{\snaipw}{\widehat{\tau}_{\textsc{AIPW}}}
\newcommand{\tnaipw}{\widehat{\psi}_{\textsc{AIPW}}}
\newcommand{\vipw}{\mathrm{V}^{\mathrm{IPW}}}
\newcommand{\vaipw}{\mathrm{V}^{\mathrm{AIPW}}}
\newcommand{\vipwstrong}{\mathrm{V}^{\text{IPW}}_{\text{strong}}}
\newcommand{\vipwweak}{\mathrm{V}^{\text{IPW}}_{\text{weak}}}
\newcommand{\vaipwstrong}{\mathrm{V}^{\text{AIPW}}_{\text{strong}}}
\newcommand{\vaipwweak}{\mathrm{V}^{\text{AIPW}}_{\text{weak}}}

\newcommand{\vipwefron}{V^{\text{IPW}}_{\text{Efron}}}
\newcommand{\vaipwefron}{V^{\text{AIPW}}_{\text{Efron}}}
\newcommand{\vipwwei}{V^{\text{IPW}}_{\text{Wei}}}
\newcommand{\vaipwwei}{V^{\text{AIPW}}_{\text{Wei}}}

\newcommand{\ze}{\zeta}
\newcommand{\vipwstronghat}{V^{\text{IPW}}_{\text{strong}}}
\newcommand{\vipwweakhat}{V^{\text{IPW}}_{\text{weak}}}
\newcommand{\vaipwstronghat}{V^{\text{AIPW}}_{\text{strong}}}
\newcommand{\vaipwweakhat}{V^{\text{AIPW}}_{\text{weak}}}
\newcommand{\widetildepest}{\overline{p}}
\newcommand{\m}{m}
\newcommand{\Ybar}{\widebar{Y}}
\newcommand{\Yvar}{\sigma^2}
\newcommand{\Ycov}{\sigma_{01}} 
\newcommand{\moment}{m^2}
\newcommand{\crossmoment}{m_{01}}

\newcommand{\pstar}{p^\star}
\newcommand{\ponestar}{p_1^\star}
\newcommand{\ptwostar}{p_2^\star} 
\newcommand{\mhat}{\widehat{m}}

\theoremstyle{plain}
\newtheorem{theorem}{Theorem}

\newtheorem{lemma}{Lemma}
\newtheorem{definition}{Definition}

\newtheorem{remark}{Remark}

\newtheorem{assumption}{Assumption}
\colorlet{shadecolor}{gray!10}
\colorlet{lightyellow}{yellow!10}

\title{\LARGE \textbf{Design Stability in Adaptive Experiments: \\ Implications for Treatment Effect Estimation}}
\author{
\large 
Saikat Sengupta$^{\dagger}$\\
\normalsize \href{mailto:mb2419@isical.ac.in}{mb2419@isical.ac.in} \\
\and
Koulik Khamaru$^{\ddagger}$\\
\normalsize \href{mailto:kk1241@stat.rutgers.edu}{kk1241@stat.rutgers.edu} \\
\and
Suvrojit Ghosh$^{\ddagger}$\\
\normalsize \href{mailto:sg1565@scarletmail.rutgers.edu}{sg1565@scarletmail.rutgers.edu} \\
\and
Tirthankar Dasgupta$^{\ddagger}$\\
\normalsize \href{mailto:tirthankar.dasgupta@rutgers.edu}{tirthankar.dasgupta@rutgers.edu} \\[0.3cm]
\normalsize $^{\dagger}$ Indian Statistical Institute, Kolkata \quad
\normalsize $^{\ddagger}$ Department of Statistics, Rutgers University
}
\date{\today}

\begin{document}
\maketitle

\begin{abstract}
We study the problem of estimating the average treatment effect (\textsc{ATE}) under sequentially adaptive treatment assignment mechanisms. In contrast to classical completely randomized designs, we consider a setting in which the probability of assigning treatment to each experimental unit may depend on prior assignments and observed outcomes. Within the potential outcomes framework~\cite{Neyman1923}, we propose and analyze two natural estimators for the \textsc{ATE}: the inverse propensity weighted (IPW) estimator and an augmented IPW (AIPW) estimator. The cornerstone of our analysis is the concept of \emph{design stability}, which requires that as the number of units grows, either the assignment probabilities converge, or sample averages of the inverse propensity scores and of the inverse complement propensity scores converge in probability to fixed, non-random limits. Our main results establish central limit theorems for both the IPW and AIPW estimators under \emph{design stability} and provide explicit expressions for their asymptotic variances. We further propose estimators for these variances, enabling the construction of asymptotically valid confidence intervals. Finally, we illustrate our theoretical results in the context of Wei’s adaptive coin design~\cite{Wei1978-gd} and Efron’s biased coin design~\cite{Efron1971-yh}, highlighting the applicability of our results to sequential experimental designs with adaptive randomization.
\end{abstract}


\section{Introduction}
\label{sec:intro}
Estimating the average treatment effect is a foundational problem in causal inference, especially when evaluating interventions in fields such as healthcare~\cite{Shen2020-zr}, education~\cite{Kaplan2016-ok}, public policy~\cite{Finkelstein2020-jw}, development economics~\cite{Imbens2004-nc}, and digital experimentation~\cite{Shi2023-at}. Traditional methods often assume simple randomized designs with independent and identically distributed (i.i.d.) units and fixed treatment assignment probabilities. However, many real-world experiments depart from this idealized setting: units often arrive sequentially, treatment assignments may adapt over time based on previous allocations, observed outcomes, or covariate information, and the study population is finite. Such sequential finite-population setups are common in adaptive clinical trials, online A/B testing, and policy evaluations. In these scenarios, adaptively assigning treatments can lead to complex dependencies among units, causing traditional ATE estimators to become biased or inefficient and undermining the applicability of conventional asymptotic results. Despite its practical importance, this setting remains methodologically less explored. 

This paper develops a general framework for ATE estimation and inference under sequential designs in finite populations. We study the asymptotic behavior of the average treatment effect estimators under a class of sequential Bernoulli assignment mechanisms, in which the probability of assigning treatment to unit $i$ may depend on the observed history up to that point. 
Two classic examples of such assignments are Wei's adaptive coin design~\cite{Wei1978-gd}, and Efron's biased coin design~\cite{Efron1971-yh}. More general classes of adaptive assignments in clinical trials can be found in \cite{bb691c50-25c9-325f-b1ee-0cc92167cbd1}. 

Our contributions proceed in four parts. First, we analyze the \emph{inverse propensity weighted} (IPW) estimator for the average treatment effect and propose an improvement via a finite-population analogue of its \emph{augmented version} (AIPW), commonly used in model-based frameworks. Second, for both estimators, we establish central limit theorems under general sequential designs that satisfy a newly defined property called \emph{design stability}. Third, under two different forms of \emph{design stability - strong and weak} - we derive estimators of the asymptotic variances of the average treatment effect estimators. These variance estimators, and the corresponding confidence intervals for the ATE, are conservative in that they are asymptotically positively biased, leading to asymptotic overcoverage of the confidence intervals. However, the biases vanish under certain forms of treatment effect homogeneity, yielding correct asymptotic coverage of the confidence intervals. Finally, we specialize these results to the two concrete experimental designs mentioned above, arguing that one of them (Wei's adaptive coin design~\cite{Wei1978-gd}) satisfies the \emph{strong design stability} condition, whereas the other (Efron's biased coin design~\cite{Efron1971-yh}) satisfies the \emph{weak design stability} condition.

The remainder of the paper is organized as follows. Section~\ref{sec:related_work} reviews relevant prior work. Section~\ref{sec:problem_desc} formally defines the problem, introduces the potential outcomes framework, describes the sequential assignment structure, and presents the estimators of interest. Section~\ref{sec:results} presents the main theoretical results, including central limit theorems for adaptive designs and conservative asymptotic variance estimators for confidence interval construction. These results are then specialized in Section~\ref{sec:application} to two widely used adaptive treatment assignment mechanisms. Sections~\ref{sec:wei} and~\ref{sec:efron} examine the \emph{stability} of these designs and present simulation studies illustrating the finite-sample performance of the proposed estimators and supporting the theoretical findings. Section~\ref{sec:clinical_trial} briefly discusses extensions to more general adaptive clinical trial designs, while Section~\ref{sec:discussion} concludes with a discussion and directions for future research. Proofs of all main and auxiliary results are provided in the Supplementary Material.


\section{Related work}
\label{sec:related_work}
Causal inference in experimental settings is typically framed within two paradigms. The infinite superpopulation perspective views study units as random draws from an underlying population, with randomness arising from the data-generating process. In contrast, the finite-population or design-based perspective treats the set of units as fixed, with uncertainty introduced solely through the experimental design. This latter view, combined with the potential outcomes formulation, traces back to~\cite{Neyman1923}, who conceptualized each unit’s treatment and control responses as fixed quantities and attributed randomness entirely to randomization. While classical asymptotic theory~\cite[e.g.,][]{lehmann1999, vandervaart2000} is often aligned with the superpopulation framework, many applications, particularly randomized trials and survey sampling, are more naturally analyzed from the finite-population perspective~\cite[e.g.,][]{fisher1935, cochran1977, rosenbaum2002, imbens2015}.

The study of asymptotic normality in causal inference can be traced back to results on simple random sampling. Classical central limit theorems were established by~\cite{madow1948},~\cite{erdos1959},~\cite{hajek1960}, with convenient formulations presented in~\cite{lehmann1975, lehmann1999}. These sampling-based central limit theorems can also be viewed as special cases of the more general results for rank statistics~\cite{wald1944, Noether1949-br, Fraser1956-qf, hajek1961}. Further foundational work includes the theory of U-statistics developed by~\cite{Hoeffding1963-tc} and the weak convergence results of~\cite{1177692698}, which laid the groundwork for modern asymptotic theory in survey sampling and experimental design. Because treatment and control groups in randomized experiments correspond to simple random samples from the finite set of experimental units, these sampling-based central limit theorems are directly applicable to the difference-in-means estimator of the average treatment effect. This connection underlies much of the early asymptotic justification in randomization-based causal inference~\cite[e.g.,][]{Liu2014-kk, Ding2018-lr, Fogarty2018, li2017general}. 

In modern applications such as adaptive clinical trials~\cite{rosenberger2016, Hu2006-vc, bb691c50-25c9-325f-b1ee-0cc92167cbd1}, online A/B tests~\cite{article}, and adaptive policy experiments~\cite{ATHEY201773}, treatment assignments may evolve in response to interim data, violating the independence assumptions of static designs. Such sequential mechanisms introduce dependence across units, requiring martingale‐based techniques~\cite[Chapter 3]{Hall1980-ny} in place of classical i.i.d. arguments.

There is a growing literature on inference under adaptive treatment assignment and bandit-style designs adopts stability-type assumptions and develops martingale-based central limit theorems, leading to methodological parallels with our work; examples include \cite{Kato2020-qo}, \cite{pmlr-v236-cook24a}, \cite{NEURIPS2021_eff30581}, \cite{Hadad2021-jg}, and \cite{NEURIPS2020_6fd86e0a, Zhang2021-zy}. There are, however, several fundamental differences. First, all of these works operate under a superpopulation framework, where randomness in observed outcomes arises from assumed probabilistic outcome models, whereas we adopt a finite-population perspective in which the sole source of randomness is the known probabilistic treatment assignment mechanism. This distinction leads to substantively different assumptions and technical arguments. Second, while some of the above papers (\cite{Hadad2021-jg}; \cite{Zhang2021-zy}) explicitly use a potential outcomes framework, others do not. Third, despite superficial similarities in the notion of ``stability,” the bandit algorithms studied in \cite{NEURIPS2020_6fd86e0a, Zhang2021-zy} are designed to eventually commit to a single arm, in sharp contrast to our setting, where long-run balance across treatment arms, which is central to randomized experimentation, is the primary objective.

The present work addresses this gap, establishing central limit theorems for IPW and AIPW-inspired estimators in finite-population under broad sequential designs in a purely design-based inferential framework, where the potential outcomes are assumed fixed.



\section{Problem Description}
\label{sec:problem_desc}
Consider a study with $N$ experimental units, indexed by $i = 1, \ldots, N$. We adopt the potential outcomes framework, introduced by \cite{Neyman1923} and later formalized by \cite{Rubin1974-gr}. For the $i$th unit, the outcome of interest $Y_i$ is characterized by two potential outcomes: $\Yizero$ under control and $\Yione$ under treatment. The individual treatment effect is defined as $\taui = \Yione - \Yizero$, and our target parameter is the average treatment effect (\textsc{ATE}), defined as
\begin{align}
\label{eqn:ATE}
 \ate = \frac{1}{N} \sum_{i=1}^N \taui \qquad \qquad (\textsc{ATE}).
\end{align}

\medskip
Assumptions of homogeneity of unit-level treatment effects $\tau_1, \ldots, \tau_N$ play important roles in finite-population causal inference. For example, the assumption that the $\tau_i$'s are the same for $i=1, \ldots, N$, or equivalently,
\begin{align}
\label{eqn:additive-model}
    \Yione = \Yizero + \tau \quad \text{for all } i,
\end{align}
for some constant $\tau \in \mathbb{R}$, is referred to as \emph{additivity} of potential outcomes and is standard in literature~\cite[e.g.,][Chapter 6]{imbens2015}. Here we introduce the following definition that generalizes the concept of treatment effect homogeneity.

\begin{definition}[Generalized treatment effect homogeneity]
\label{definition:additive-model}
Potential outcomes $(Y_i(0), Y_i(1))$, $i=1, \ldots, N$ are said to satisfy generalized treatment effect homogeneity if
\begin{align}
\label{eqn:gen-additive-model}
    \Yione - \overline{Y}_N(1) = \lambda(\Yizero - \overline{Y}_N(0)) \quad \text{for all } i,
\end{align}
for some constant $\lambda \in \mathbb{R}$, where $\overline{Y}_N(\ell) = \tfrac{1}{N}\sum_{i=1}^N Y_i(\ell)$ for $\ell \in \{0,1\}$.
\end{definition}
It is easy to see that additivity~\eqref{eqn:additive-model} implies generalized treatment effect homogeneity~\eqref{eqn:gen-additive-model}. Another sufficient condition for~\eqref{eqn:gen-additive-model} is additivity of potential outcomes on a log-scale, that is,
\begin{align}
\label{eqn:log-additive-model}
    \Yione = c \Yizero \quad \text{for all } i,
\end{align}
for some constant $c \in \mathbb{R}$. We will see that the conditions~\eqref{eqn:additive-model}-\eqref{eqn:log-additive-model} play important roles in the inference problem to be discussed.

\medskip
In the classical randomized treatment allocation design, a pre-defined constant number $N_1$ of the 
$N$ units are assigned to treatment, with the subset selected uniformly at random~\cite{Rubin1974-gr}. Formally, let $\mathbf{K} = (K_1, K_2, \ldots, K_N)^\T \in \{0,1\}^N$ denote the random assignment vector, where $K_i = 1$ if $i$th unit is assigned to the treatment group and $K_i = 0$ otherwise. A simple random sample of size $N_1$ is chosen from the finite population using the assignment vector $\mathbf{K}$, where $\Prob (\mathbf{K}=\mathbf{k}) = \binom{N}{N_1}^{-1}$ for all $\mathbf{k} \in \{0,1\}^N$ satisfying $\mathbf{1}_N^\T\mathbf{k} = N_1$. Given a treatment assignment vector $\mathbf{k} \in \{0,1\}^N$, the observed data $\{Y_i\}_{i=1}^N$ are the realized potential outcomes, where each unit’s outcome corresponds to its assigned treatment or control, defined by
\begin{align}
\label{eqn:potential_outcomes}
   Y_i =  K_i \Yione  +  (1 - K_i) \Yizero  \qquad \text{for} \qquad 
   i = 1, \ldots ,N.
\end{align} A natural estimator of the ATE under the randomized treatment assignment described above is the difference in sample means between the treatment and control groups, i.e.
\begin{align*}
    \taubar 
    &= \frac{1}{N_1} \sum_{i=1}^N K_i Y_i 
       - \frac{1}{N_0} \sum_{i=1}^N (1-K_i) Y_i, \qquad(\textsc{difference-in-means estimator})
\end{align*}
where $N_1 = \sum_{i=1}^N K_i \text{ and } N_0 = N - N_1.$  This estimator is unbiased and satisfies a central limit theorem as $N$ grows to infinity~\cite{li2017general}. We note that the difference-in-means estimator is a special case of the Horvitz–Thompson type estimator \cite{horvitz1952}, also known as the inverse propensity weighted (IPW) estimator, defined as
    \begin{align}
        \label{eqn:IPW}
        \snipw &= \frac{1}{N} \sum_{i=1}^N \left\{\frac{ K_iY_i}{p_i} -  \frac{(1 - K_i)Y_i}{1 - p_i}\right\}.
    \end{align}    
In particular, when $p_i = N_1/N$ for all $i$, this estimator reduces to the usual difference-in-means estimator.


\medskip
In contrast to the classical randomized treatment assignment, we consider in this paper a \emph{sequential treatment assignment} mechanism, in which the probability of assigning the $i$th unit to treatment is adaptive. In other words, for each unit $i \in \{1,2,\ldots, N\}$, the assignment indicator $K_i$, conditional on the past history, follows a Bernoulli distribution with success probability $p_i$, which we refer to as the \emph{inclusion probability}. The inclusion probability $p_i$ is not fixed; rather, it is a measurable function of the prior assignment history and outcomes. Formally, we define a sequence of increasing sigma-fields $\{ \sigmafield \}_{i \geq 1}$, where $\sigmafield = \sigma (K_1, Y_1, \ldots, K_{i-1}, Y_{i - 1})$ represents the cumulative information available after assigning treatment or control and observing the outcome of the $(i-1)$th unit. The assignment mechanism is such that
\begin{align}
\label{eqn:sequential-assign}
  p_i \in \sigmafield \quad \text{and} \quad   \Prob(K_i = 1 \mid \sigmafield) = p_i.  \qquad (\textsc{Sequential treatment assignment})
\end{align} In words, the probability of assigning treatment to the $i$th unit may depend on all previous assignments and observations up to stage $(i - 1)$ \emph{in an arbitrary way}; we call this assignment a \emph{sequential treatment assignment}. As an example, one may choose $p_i$ to promote relative balance between the numbers of treatment and control assignments. In this case, $p_i$ can be defined as the complement of the moving average of past assignments: 
\begin{align*}
p_i = 1 - \frac{\sum_{j=1}^{i-1} K_j}{i - 1} \qquad \text{for \;\;  $i \geq 2,$ \;\;\;\;  and} \qquad  p_1 = \frac{1}{2}. 
\end{align*}In this assignment, if the previous  units have mostly been assigned treatment, the chance of assigning treatment to the next unit will be lowered and vice versa~\cite{Wei1978-gd}. 

\medskip
In this paper, we aim to develop estimators, establish their asymptotics, and provide valid inference for the ATE under the sequential treatment assignment scheme~\eqref{eqn:sequential-assign}. 
We first consider the Horvitz–Thompson type estimator defined in~\eqref{eqn:IPW} as a natural unbiased estimator of the average treatment effect. In our setting, unlike a static completely randomized experiment, the $p_i$'s will not be equal and will depend on the past history.

It is well known that IPW estimators suffer from inflated 
variance when the probabilities approach extremes \cite{tsiatis2006}. In a model-based setting, this limitation of IPW estimator is mitigated by the AIPW estimator \cite{robins1994} via model augmentation, offering double robustness. In our setting, where no probability model for the potential outcomes is assumed, we propose the following finite-population model-free version of the AIPW estimator:

    \begin{align}
        \label{eqn:AIPW}
        \snaipw &= \frac{1}{N} \sum_{i=1}^N \left[\left\{  \frac{K_i(Y_i - \widehat{Y}_{i-1}(1)) }{p_i} + \widehat{Y}_{i-1}(1) \right\}
- \left\{\frac{(1 - K_i)(Y_i - \widehat{Y}_{i-1}(0)) }{1 - p_i} + \widehat{Y}_{i-1}(0) \right\}\right], 
    \end{align} where $\widehat{Y}_{1}(0) = \widehat{Y}_{1}(1) = 0$, and for $i \geq 2$
\begin{align}
\label{eqn:yi-hat}
   \widehat{Y}_{i-1}(0) = \frac{1}{i-1} \sum_{j = 1}^{i - 1} \frac{(1 - K_j)Y_j }{1 - p_j}, \qquad \text{and} \qquad \widehat{Y}_{i-1}(1) = \frac{1}{i-1} \sum_{j = 1}^{i - 1} \frac{K_jY_j }{p_j}.
\end{align}
Note that for large $N$, the weighted average $\widehat{Y}_{N}(\ell)$ serves as an intuitive estimator of $\overline{Y}_N(\ell)$ for $\ell \in \{0,1\}$.
In this sense, $\snaipw$ is directly motivated from the classical AIPW estimator~\cite{robins1994}. We formalize this intuition in a later theorem on the behavior of the AIPW estimator $\snaipw$ (see Theorem~\ref{Thm:aipw}).


\section{Main Results}\label{sec:results}
This section presents our main theoretical contributions. We establish the asymptotic normality of the estimators 
$\snipw$ and $\snaipw$ under the sequential experimental designs in~\eqref{eqn:sequential-assign}, 
and derive conservative variance estimators that facilitate the construction of asymptotically valid confidence intervals for the average treatment effect $\ate$. Our analysis proceeds in two steps: first, we prove central limit theorems for both estimators; 
second, we propose conservative estimators of their asymptotic variances. 
Together, these results enable the construction of asymptotically valid confidence intervals for $\ate$.

\medskip
The foundation of our analysis rests on a structural condition that we call \emph{design stability}. We consider two notions of design stability (a) \emph{strong stability}, and (b) \emph{weak stability}. Strong design stability requires that the assignment probabilities themselves converge asymptotically, ensuring that the design does not drift in the limit. Weak design stability, however, relaxes this by requiring only that the sample averages of the inverse propensity scores and of the inverse complement propensity scores converge in probability to finite, non-random limits. At a high level, both forms of stability ensure that the cumulative effect of sequential randomization does not induce excessive variability in the long run.

\begin{definition}[Strong design stability]
\label{def:design-stability}
    A sequential design with inclusion probabilities $\{ p_i \}_{i \geq 1}$ is said to be strongly stable if there exists a \emph{non-random} scalar $\pstar \in (0,1)$ such that
    \begin{align}
        p_i \stackrel{\mathbb{P}}{\rightarrow} \pstar.
    \end{align}
\end{definition}


Although the notion of strong design stability is intuitive, it is not satisfied by several popular designs. A concrete example is Efron’s biased coin design~\cite{Efron1971-yh}, which enforces balance between treatment and control assignments. Fortunately, Definition~\ref{def:design-stability} can be relaxed so that, even if a design is not stable in the strong sense, central limit theorems for the IPW and AIPW estimators may still hold under weaker regularity conditions. This motivates the following notion of weak stability.


\begin{definition}[Weak design stability]
\label{def:weak-design-stability}
    A sequential design with inclusion probabilities $\{ p_i \}_{i \geq 1}$ is said to be weakly stable if there exists \emph{non-random} scalars $\ponestar, \ptwostar \in (0,1)$ such that
    \begin{align}
\label{eqn:modified_stability}
\frac{1}{N} \sum_{i = 1}^N \frac{1}{p_i} \stackrel{\mathbb{P}}{\rightarrow} \frac{1}{\ponestar} \qquad \text{and} \qquad \frac{1}{N} \sum_{i = 1}^N \frac{1}{1 - p_i} \stackrel{\mathbb{P}}{\rightarrow} \frac{1}{1 - \ptwostar}.
\end{align}
\end{definition}


\begin{remark}
Technically, the inclusion probabilities can be viewed as a triangular array 
\(\{p_{i,N} : 1 \le i \le N,\; N \ge 1\}\), which induces a sequence of estimands 
\(\{\bar{\tau}_N\}_{N \ge 1}\) and, a sequence of estimators 
\(\{\widehat{\tau}_N\}_{N \ge 1}\). Analysis for a fixed \(N\) is referred to as finite-sample, 
whereas statements concerning the limiting behavior as \(N \to \infty\) are termed asymptotic. Although our framework is asymptotic, most of our results are derived from finite-sample arguments 
and are interpreted through this asymptotic lens. For notational simplicity, we suppress the 
dependence on \(N\) whenever no ambiguity arises.
\end{remark}

Suppose strong stability holds, Since, by assumption, the inclusion probabilities are bounded away from both \(0\) and \(1\), Slutsky’s theorem implies
$\frac{1}{p_i} \stackrel{\mathbb{P}}{\rightarrow} \frac{1}{\pstar}$ and
$\frac{1}{1-p_i} \stackrel{\mathbb{P}}{\rightarrow} \frac{1}{1-\pstar}$
    and hence,
    $\frac{1}{N}\sum_{i=1}^N \frac{1}{p_i} \stackrel{\mathbb{P}}{\rightarrow} \frac{1}{\pstar}$
and $\frac{1}{N}\sum_{i=1}^N \frac{1}{1-p_i} \stackrel{\mathbb{P}}{\rightarrow} \frac{1}{1-\pstar}$,
    which shows that strong stability implies weak stability with \(\ponestar = \ptwostar = \pstar\). The tradeoff between these two stability conditions becomes apparent when estimating the asymptotic variance of our estimators to construct confidence intervals for the ATE. While variance estimators can be constructed in a completely data-dependent manner under strong stability, weak stability requires additional restrictions (see Theorems~\ref{corr:vipw_est_weak} and~\ref{corr:vaipw_est_weak} for more details). 


\subsection{The IPW estimator}
\label{sec:ipw_results}
We now turn to the asymptotic behavior of the IPW estimator~$\snipw$, as defined in~\eqref{eqn:IPW}. To derive our main result, we impose a positivity condition on the inclusion probabilities along with uniform boundedness and natural moment conditions on the potential outcomes. 
\begin{assumption}
\label{assn:IPW}
The inclusion probabilities and potential outcomes satisfy the following regularity conditions:
\begin{enumerate}[label=(\alph*)]
    \item There exists $\delta \in (0,1)$ such that $p_i \in [\delta, 1 - \delta]$ for all $i \geq 1$.     
    \item There exists a constant $M > 0$ such that 
    \begin{align*}
        |Y_i(\ell)| \leq M \quad \text{for all } i \geq 1 \text{ and } \ell \in \{0,1\}.
    \end{align*}
    
    \item  The following limits exist:
    \begin{align*}
        \lim_{N \to \infty} \frac{1}{N} \sum_{i=1}^N \Yizero^2 = \moment_{0},   \quad 
        \lim_{N \to \infty} \frac{1}{N} \sum_{i=1}^N \Yione^2 = \moment_{1},  \quad \text{and} \quad
        \lim_{N \to \infty} \frac{1}{N} \sum_{i=1}^N \Yizero \Yione = \crossmoment,
    \end{align*}
    where $\moment_{0}, \moment_{1} > 0$ and $\crossmoment \in \mathbb{R}$.
\end{enumerate}
\end{assumption} 
The first condition in Assumption~\ref{assn:IPW} ensures that the IPW estimator~\eqref{eqn:IPW} is well defined, the second condition is a uniform bound on the potential outcomes, and the third assumption ensures that the limiting asymptotic variance of $\snipw$ exists. With this set-up, we have the following guarantees on the asympotic behavior of $\snipw$.

\begin{theorem}
\label{Thm:ipw}
Suppose Assumption~\ref{assn:IPW} holds, and the sequential design with inclusion probabilities $\{p_i\}_{i \geq 1}$ is either strongly or weakly stable in the sense of Definition~\ref{def:design-stability} or Definition~\ref{def:weak-design-stability}, respectively. Then the IPW estimator~\eqref{eqn:IPW} satisfies
\begin{equation}
\sqrt{N}\left(\snipw - \ate\right) \xrightarrow{d} \mathcal{N}\left(0, \vipw\right),
\label{cltipw}
\end{equation}
with asymptotic variance
\begin{align}
\vipw =
\begin{cases}
\vipwstrong = \moment_{0}\dfrac{\pstar}{1-\pstar} + \moment_{1}\dfrac{1-\pstar}{\pstar} + 2\crossmoment 
& \text{(strong design stability)}, \\[1.2em]
\vipwweak = \moment_{0}\dfrac{\ptwostar}{1-\ptwostar} + \moment_{1}\dfrac{1-\ponestar}{\ponestar} + 2\crossmoment 
& \text{(weak design stability)}.
\end{cases}
\label{eqn:vipw_cases}
\end{align}
\end{theorem}

\begin{remark}
The proof of Theorem~\ref{Thm:ipw} proceeds by rewriting the centered and scaled IPW estimator as a sum of martingale difference terms. This representation allows us to apply the martingale central limit theorem \cite[Chapter 3]{Hall1980-ny}. We establish unbiasedness by verifying that each summand has zero conditional mean, and then compute the conditional variance, which converges under both strongly stable and weakly stable designs to the stated asymptotic variance. Finally, we check the Lindeberg condition, ensuring that the contribution of large deviations vanishes. Together, these steps yield asymptotic normality of the IPW estimator with the asymptotic variance given in~\eqref{eqn:vipw_cases}. Refer to Section~\ref{sec:Proof-of-Thm:ipw} of the Supplementary Material for detailed proof. 
\end{remark}

\begin{remark}
It is worth noting that our stability definitions provide two distinct mechanisms to ensure that the random conditional variance converges in probability to a deterministic limit. As seen in the proof of Theorem \ref{Thm:ipw}, this conditional variance is 
\begin{align*}
    V_N = \frac{1}{N}\sum_{i=1}^N\frac{p_i}{1-p_i}\Yizero^2 + \frac{1}{N}\sum_{i=1}^N\frac{1-p_i}{p_i}\Yione^2 + \frac{2}{N}\sum_{i=1}^N\Yizero\Yione.
\end{align*}

Strong stability requires $p_i$ to converge in probability to fixed limit $p^\star$. Intuitively, the design asymptotically ``freezes'' as every term in the sum $V_N$ becomes constant. On the other hand, weak stability permits $p_i$ to fluctuate but requires the sample means of $N^{-1}\sum_{i=1}^N p_i^{-1}$ and $N^{-1}\sum_{i=1}^N (1-p_i)^{-1}$ to converge in probability to some fixed limits $1/\ponestar$ and $1/\ptwostar$. Intuitively, this acts as a Law of Large Numbers for the variance process. While individual weights $p_i/(1-p_i)$ and $(1-p_i)/p_i$ remain volatile, their average converges settles to a fixed constant. This ensures that the ``total precision'' of the experiment becomes predictable, satisfying the martingale CLT conditional variance stabilization condition without requiring the assignment algorithm to commit to a specific allocation.
\end{remark}

\medskip
Having established the asymptotic normality of the IPW estimator, we next construct confidence intervals for the average treatment effect $\ate$. This, in turn, requires estimation of the asymptotic variance $\vipw$. First, we consider that the design is strongly stable in the sense of Definition~\ref{def:design-stability}, and assume that $p^\star$ is known (which is the case in our illustrative example on strongly stable designs). To estimate $\vipwstrong$ in~\eqref{eqn:vipw_cases}, we must estimate $\moment_0$, $\moment_1$, and $\crossmoment$. While obtaining consistent estimators of $\moment_0$ and $\moment_1$ under strong stability is straightforward, the cross-moment term $\crossmoment$ cannot be estimated from the observed outcomes without additional assumptions, as only one potential outcome is observed for each unit. To address this problem, we apply the Cauchy–Schwarz inequality to obtain $|\crossmoment| \leq \m_{0} \m_{1}$, leading to the conservative variance estimator:
\begin{align}
\widehat{\vipwstronghat} = \left(\widehat{\m}_0 \sqrt{\frac{p^\star}{1-p^\star}} + \widehat{\m}_1 \sqrt{\frac{1-p^\star}{p^\star}}\right)^2,
\label{eqn:vipw_liberal}
\end{align}
where $\widehat{\m}_0$ and $\widehat{\m}_1$ are estimators of $m_0$ and $m_1$ that are consistent under strong stability. We propose the following intuitive estimators for $\moment_0$ and $\moment_1$
\begin{align}
\label{eqn:mhats}
\widehat{m}^2_{0} = \frac{1}{\max\{N_0,1\}}\sum_{i=1}^N (1-K_i)Y_i^2,\qquad \text{and} \qquad
\widehat{m}^2_{1} = \frac{1}{\max\{N_1,1\}}\sum_{i=1}^N K_iY_i^2, \qquad    
\end{align}
where $N_1 = \sum_{i = 1}^N K_i$ and $N_0 = N -N_1$. The variance estimator in~\eqref{eqn:vipw_liberal}, which incorporates the estimators $\widehat{m}_{0}^2$ and $\widehat{m}_{1}^2$ defined in~\eqref{eqn:mhats}, is consistent when the potential outcomes are additive on a log-scale, that is, satisfy \eqref{eqn:log-additive-model} and has an asymptotic positive bias otherwise. Hence, we obtain the following theorem.





\begin{theorem}\label{corr:vipw_est}
For strongly stable designs (Definition~\ref{def:design-stability}), the estimators $\widehat{m}_0^2$ and $\widehat{m}_1^2$ defined in~\eqref{eqn:mhats} 
are consistent for $\moment_{0}$ and $\moment_{1},$ respectively. Furthermore, the variance estimator $\widehat{\vipwstronghat}$ given by~\eqref{eqn:vipw_liberal} provides a conservative estimate of $\vipwstrong$, 
and is consistent when the potential outcomes are additive on a log scale, that is, satisfy~\eqref{eqn:log-additive-model}. 
\end{theorem}
Refer to Section~\ref{sec:Proof-of-Corr:vipw_est} of the Supplementary Material for a proof of the theorem. 

\begin{remark}
\label{remark:phat}
If for a strongly stable design the limiting value $\pstar$ is difficult to compute explicitly, the following consistent estimator
\begin{align}
\label{eqn:phat}
\widehat{p}^\star = \frac{1}{N}\sum_{i=1}^N p_i, \qquad  
\end{align}
may be substituted for $p^\star$ into the variance estimator $\widehat{\vipwstronghat}$, as defined in~\eqref{eqn:vipw_liberal}. 
\end{remark}

\begin{remark}
\label{rem:hw}
$\widehat{\vipwstronghat}$ closely resembles the conservative variance estimator $\widehat{V}'$ for the variance of the difference-in-means estimator~\cite[Homework Problem 4.5]{Ding2023-ku}. Under strong stability, the effective sample sizes satisfy $N_1 \approx Np^\star$ and $N_0 \approx N(1-p^\star)$ asymptotically. Substituting these approximations into the variance estimator $\widehat{V}'$ in~\cite[Homework Problem 4.5]{Ding2023-ku} yields the same functional form as our conservative variance estimator.
\end{remark}

We now turn to weakly stable designs (Definition~\ref{def:weak-design-stability}) and assume that the limiting quantities $\ponestar$ and $\ptwostar$ are known (which is the case in our illustrative example on weakly stable designs). Using arguments exactly analogous to the strongly stable case, we obtain the following conservative estimator of $\vipwweak$~\eqref{eqn:vipw_cases}:

\begin{align}
\widehat{\vipwweak} = \widetilde{\m}_0^2 {\frac{\ptwostar}{1-\ptwostar}} + \widetilde{\m}_1^2 {\frac{1-\ponestar}{\ponestar}} + 2\widetilde{\m}_0\widetilde{\m}_1,
\label{eqn:varipw_est_weak}
\end{align}
where $\widetilde{m_0}$ and $\widetilde{m_1}$ are estimators of $\m_0$ and $\m_1$ that are consistent under weak stability. However, unlike strongly stable designs, it is difficult to consistently estimate $\moment_0$ and $\moment_1$ without further restrictions. This illustrates the tradeoff between strong and weak design stability: although weak stability enlarges the class of admissible designs, it requires additional assumptions for consistent estimation of $\moment_0$ and $\moment_1$, and hence for conservative variance estimation.
In particular, under the mild assumption (though not necessarily minimal) that, for some constant $\widetilde{p} \in (0,1)$,
\begin{align}
    \frac{1}{N}\sum_{i=1}^N p_i \stackrel{\mathbb{P}}{\rightarrow} \widetilde{p},
\label{eqn:extra_assumption}
\end{align}
a consistent estimator of $\moment_0$ is 
\begin{align}
    \widetilde{m}_{0}^2 = \frac{1}{N (1-\widetilde{p})}\sum_{i=1}^N (1-K_i) Y_i^2,
    \label{eqn:weak_m0_ipw}
\end{align}
with the analogous estimator for $\moment_1$ given by 
\begin{align}
    \widetilde{m}_{1}^2 = \frac{1}{N \widetilde{p}}\sum_{i=1}^N K_i Y_i^2.
    \label{eqn:weak_m1_ipw}
\end{align}
Some intuition and an example illustrating when weak stability holds but \eqref{eqn:extra_assumption} does not, is provided in Section~\ref{sec:weak_stab_alone_doesnt_imply_eqn18} of the supplementary material. 
As in the strongly stable case, the variance estimator in~\eqref{eqn:varipw_est_weak}, which incorporates the estimators~$\widetilde{m}_{0}^2$ and~$\widetilde{m}_{1}^2$ defined in~\eqref{eqn:weak_m0_ipw} and~\eqref{eqn:weak_m1_ipw}, respectively, is consistent when the potential outcomes are additive on the log-scale; that is, when they satisfy~\eqref{eqn:log-additive-model}. Thus, we have the following theorem.

\begin{theorem}\label{corr:vipw_est_weak}
For weakly stable designs (Definition~\ref{def:weak-design-stability}), under the sufficient condition~\eqref{eqn:extra_assumption}, the estimators $\widetilde{m}_0^2$ and $\widetilde{m}_1^2$ from~\eqref{eqn:weak_m0_ipw} and~\eqref{eqn:weak_m1_ipw} are consistent for $\moment_0$ and $\moment_1$, respectively. 
Furthermore, the variance estimator $\widehat{\vipwweakhat}$ given by \eqref{eqn:varipw_est_weak} provides a conservative estimate of $\vipwweakhat$, 
and is consistent when the potential outcomes are additive on a log scale, that is, satisfy~\eqref{eqn:log-additive-model}.
\end{theorem}
Refer to Section~\ref{sec:Proof-of-Corr:vipw_est_weak} of the Supplementary Material for a proof of the theorem. 

\begin{remark}
\label{remark:ipw_weak_remark}
If, for a weakly stable design, the limiting values \(\ponestar, \ptwostar,\) and \(\widetilde{p}\) are unknown or difficult to compute explicitly, we need to estimate them. Since $p_i \in \sigmafield$, the inclusion probability is deterministically known to the experimenter given the history. We therefore propose the following intuitive and consistent estimators for $\ponestar, \ptwostar$ and $\widetilde{p}$: 
\begin{equation}
\widehat{p}_1^\star = \frac{1}{\tfrac{1}{N}\sum_{i=1}^N \tfrac{1}{p_i}}, \qquad
\widehat{p}_2^\star = 1 - \frac{1}{\tfrac{1}{N}\sum_{i=1}^N \tfrac{1}{1-p_i}} \qquad \text{and} \qquad \widetildepest = \frac{1}{N}\sum_{i=1}^N p_i.
\label{eqn:weak_stab_consistent_ponetwo}
\end{equation}
Consistency of these estimators under weak stability is established in Section~\ref{sec:Proof-of-Corr:vipw_est_weak} of the Supplementary Material (as a part of proof of Theorem \ref{corr:vipw_est_weak}).
Substitution of estimators \eqref{eqn:weak_stab_consistent_ponetwo} into \eqref{eqn:varipw_est_weak} yields the following conservative estimator of $\vipw$ under weak stability when $\ponestar, \ptwostar$ and $\widetilde{p}$ are unknown:
\begin{align}
    \widetilde{\vipwweakhat} = \widetilde{m}_{0}^2 \left(\frac{1}{N}\sum_{i=1}^N \frac{1}{1-p_i} - 1 \right) + \widetilde{m}_{1}^2 \left(\frac{1}{N}\sum_{i=1}^N \frac{1}{p_i} - 1 \right) + 2 \widetilde{m}_{0}\widetilde{m}_{1},
\label{eqn:varipw_est_weak_unknown}
\end{align}
where $\widetilde{m}_{0}^2 = \displaystyle \frac{1}{\sum_{i=1}^N(1-p_i)}\sum_{i=1}^N (1-K_i) Y_i^2$ and $\widetilde{m}_{1}^2 = \displaystyle \frac{1}{\sum_{i=1}^N p_i}\sum_{i=1}^N K_i Y_i^2.$
\end{remark}

Since we have constructed conservative variance estimators of $\vipw$, for any target level $\alpha \in (0,1)$, an asymptotically conservative confidence interval for $\ate$, that is, one with coverage at least $(1-\alpha)$ can be constructed as
\begin{align*} 
\lim_{N\rightarrow \infty} \Prob \left( \ate \in \left[ \snipw - z_{1 - \alpha/2}  \sqrt{\frac{\widehat{V}}{N}}, \;\; \snipw + z_{1 - \alpha/2}  \sqrt{\frac{\widehat{V}}{N}} \right] \right) \geq 1 - \alpha, 
\end{align*}
where $\widehat{V} = \widehat{\vipwstronghat}$ or $\widehat{\vipwweakhat}$, and 
$z_{1 - \alpha/2}$ denotes the $(1 - \alpha/2)$th quantile of the standard normal distribution. 
Moreover, if the potential outcomes satisfy the log-additive treatment-effect~\eqref{eqn:log-additive-model}, 
the inequality holds with equality, yielding asymptotically exact coverage.

%

\subsection{The AIPW-type estimator}
\label{sec:aipw_results}
We now analyze the asymptotic behavior of the AIPW estimator $\snaipw$, as defined in~\eqref{eqn:AIPW}.
\begin{assumption}
\label{assn:AIPW}
The inclusion probabilities and potential outcomes satisfy the following regularity conditions:
\begin{enumerate}[label=(\alph*)]
    \item There exists $\delta \in (0,1)$ such that $p_i \in [\delta, 1 - \delta]$ for all $i \geq 1$.     
    \item There exists a constant $M > 0$ such that 
    \begin{align*}
        |Y_i(\ell)| \leq M \quad \text{for all } i \geq 1 \text{ and } \ell \in \{0,1\}.
    \end{align*}
    \item  The following limits exist for $\ell \in \{0,1\}:$
\begin{subequations}
\begin{align*}
\lim_{N \to \infty} \Ybar_{N}(\ell) = \Ybar_{\ell}, \qquad \lim_{N \to \infty} \frac{1}{N} \sum_{i=1}^N \left( Y_i(\ell) - \Ybar_{N}(\ell) \right)^2 = \Yvar_{\ell}, \\  \text{and} \qquad \lim_{N \to \infty} \frac{1}{N} \sum_{i=1}^N \left( \Yizero - \Ybar_{N}(0) \right)\left( \Yione - \Ybar_{N}(1) \right) = \Ycov,
\end{align*}
\end{subequations}
where $\Yvar_{0}, \Yvar_{1} >0$ and $\Ybar_{0}, \Ybar_{1}, \Ycov \in \mathbb{R}$.
\end{enumerate}
\end{assumption}
The first two conditions in Assumption~\ref{assn:AIPW} are same as that of Assumption~\ref{assn:IPW}. The last condition above ensures that the limiting variance of the AIPW estimator exists.  With this set-up, we state the asymptotic behavior of the AIPW estimator $\snaipw$.

\begin{theorem}
\label{Thm:aipw}
Suppose Assumption~\ref{assn:AIPW} holds, and the sequential design with inclusion probabilities $\{p_i\}_{i \geq 1}$ is either strongly or weakly stable in the sense of Definition~\ref{def:design-stability} or Definition~\ref{def:weak-design-stability}, respectively. Then the AIPW estimator~\eqref{eqn:AIPW} satisfies
\begin{equation}
\sqrt{N}\left(\snaipw - \ate\right) \xrightarrow{d} \mathcal{N}\left(0, \vaipw\right),
\label{cltipw}
\end{equation}
with asymptotic variance
\begin{align}
\vaipw =
\begin{cases}
\vaipwstrong = \Yvar_{0}\dfrac{\pstar}{1-\pstar} + \Yvar_{1}\dfrac{1-\pstar}{\pstar} + 2\Ycov & \text{under strong design stability}, \\[1.2em]
\vaipwweak = \Yvar_{0}\dfrac{\ptwostar}{1-\ptwostar} + \Yvar_{1}\dfrac{1-\ponestar}{\ponestar} + 2\Ycov & \text{under weak design stability}.
\end{cases}
\label{eqn:vaipw_cases}
\end{align}
\end{theorem}


\begin{remark}
The proof of Theorem~\ref{Thm:aipw} requires a different strategy from that of Theorem~\ref{Thm:ipw}, since the AIPW estimator is not directly amenable to martingale central limit theorem. To handle this, we introduce a proxy estimator, which is interesting in its own right, that is analytically more tractable and can be expressed as sum of a martingale difference sequence, allowing the martingale central limit theorem to establish its asymptotic normality. The key step is then to show that the difference between the proxy and the actual AIPW estimator is asymptotically negligible, using variance bounds and H\'{a}jek's lemma (see Lemma~\ref{lemma:hajek_proof} of the Supplementary Material). This ensures that the asymptotic distribution of the AIPW estimator coincides with that of the proxy, yielding the stated central limit theorem with variance given in Theorem~\ref{Thm:aipw}. Refer to Section~\ref{sec:Proof-of-Thm:aipw} of the Supplementary Material for detailed proof of the above theorem.
\end{remark}

\begin{remark}
\label{rem:aipwgipw}
Note that $\vaipw \leq \vipw$; that is, $\snaipw$ is more efficient than $\snipw$. This fact clearly establishes the superiority of the AIPW estimator over the IPW estimator in finite-population design-based inference under the adaptive assignment mechanism defined in (\ref{eqn:sequential-assign}). However, the corresponding variance estimators (defined later for AIPW estimators) may not inherit this ordering. See Section~\ref{sec:onthelack} of the Supplementary Material for a detailed discussion.
\end{remark}

\medskip
We now turn to the problem of estimating the asymptotic variance $\vaipw$. Under strong design stability (Definition~\ref{def:design-stability}) with known $\pstar$, estimation of $\vaipwstrong$ in \eqref{eqn:vaipw_cases} requires estimation of $\Yvar_0$, $\Yvar_1$ and $\Ycov$. As earlier, the covariance term $\Ycov$ depends on both potential outcomes for the same unit and therefore cannot be estimated without additional restrictions. Analogous to the estimation of $\vipw$ in Section \ref{sec:ipw_results}, we invoke Cauchy-Schwarz inequality to get $|\Ycov| \le \sigma_{0}\sigma_{1}$, yielding the following conservative estimator of $\vaipwstrong$ as follows:
\begin{align}
\widehat{\vaipwstronghat}
= \left( \widehat{\sigma}_0 \sqrt{ \frac{\pstar}{1 - \pstar}} +\widehat{\sigma}_1 \sqrt{\frac{1 - \pstar}{\pstar} }\right)^2,
\label{eqn:vaipw_liberal}
\end{align}
where $\widehat{\sigma}^2_0$ and $\widehat{\sigma}^2_1$ are estimators of $\sigma_0$ and $\sigma_1$ that are consistent under strong design stability. We propose the following estimators:
\begin{subequations}
\label{eqn:phat-aipw}
\begin{align}
\widehat{\sigma}^2_0 
  &= \frac{1}{\max\{N_0,1\}} \sum_{i=1}^N (1-K_i)\big( Y_i - \widehat{Y}_{i-1}(0) \big)^2, \label{eq:sigma0} \\[6pt]
\widehat{\sigma}^2_1 
  &= \frac{1}{\max\{N_1,1\}} \sum_{i=1}^N K_i\big( Y_i - \widehat{Y}_{i-1}(1) \big)^2,
  \label{eq:sigma1}
\end{align}
\end{subequations}
where $N_1 = \sum_{i=1}^N K_i$ and $N_0 = N - N_1$.
We set $\widehat{Y}_{1}(0) = \widehat{Y}_{1}(1) = 0$, and for $i \geq 2$ define
\begin{align*}
\widehat{Y}_{i-1}(0) = \frac{1}{i-1} \sum_{j<i} \frac{(1-K_j)Y_j}{1-p_j}, 
\qquad 
\widehat{Y}_{i-1}(1) = \frac{1}{i-1} \sum_{j<i} \frac{K_jY_j}{p_j}.  \end{align*}
The variance estimator in~\eqref{eqn:vaipw_liberal}, which incorporates the estimators in~\eqref{eqn:phat-aipw}, is consistent when the potential outcomes satisfy the generalized treatment effect homogeneity condition in Definition~\ref{definition:additive-model}. The preceding discussion leads to the following theorem.


\begin{theorem} \label{corr:vaipw_est}
For strongly stable designs (Definition~\ref{def:design-stability}), the estimators $\widehat{\sigma}^2_0$ and $\widehat{\sigma}_1^2$ defined in~\eqref{eqn:phat-aipw} 
are consistent for $\Yvar_0$ and $\Yvar_1$, respectively. 
Furthermore, the variance estimator $\widehat{\vaipwstronghat}$ given by~\eqref{eqn:vaipw_liberal} provides a conservative estimate of $\vaipwstrong$, 
and is consistent when the potential outcomes satisfy generalized treatment effect homogeneity~\eqref{eqn:gen-additive-model}.
\end{theorem}
A proof of this theorem is provided in Section~\ref{sec:Proof-of-Corr:vaipw_est} of the Supplementary Material. A similar remark as in Remark~\ref{rem:hw} applies.

%

As noted in Section~\ref{sec:ipw_results}, under weak design stability (Definition~\ref{def:weak-design-stability}) with known $\ponestar$ and $\ptwostar$, variance estimation is not straightforward, as additional conditions are required for the consistent estimation of $\Yvar_0$ and $\Yvar_1$. 
As before, under the additional assumption~\eqref{eqn:extra_assumption} and with known~$\widetilde{p}$, the variance components~$\Yvar_0$ and~$\Yvar_1$ can be consistently estimated by:
\begin{align}
\widetilde{\sigma}^2_0 
  &= \frac{1}{N(1-\widetilde{p})} \sum_{i=1}^N (1-K_i)\big( Y_i - \widehat{Y}_{i-1}(0) \big)^2, \label{eq:sigma0_aipw_weak} \\[6pt]
\widetilde{\sigma}^2_1 
  &= \frac{1}{N \widetilde{p}} \sum_{i=1}^N K_i\big( Y_i - \widehat{Y}_{i-1}(1) \big)^2. \label{eq:sigma1_aipw_weak}
\end{align}
As discussed previously, the cross-moment term $\Ycov$ cannot be estimated without additional assumptions, since it depends on both potential outcomes for any unit. We therefore bound it from above using the Cauchy–Schwarz inequality. Consequently, asymptotic variance $\vaipwweak$ can be conservatively estimated by:
\begin{align}
\widehat{\vaipwweakhat} = \widetilde{\sigma}_0^2 {\frac{\ptwostar}{1-\ptwostar}} + \widetilde{\sigma}_1^2 {\frac{1-\ponestar}{\ponestar}} + 2\widetilde{\sigma}_0\widetilde{\sigma}_1.
\label{eqn:varaipw_est_weak}
\end{align}
As in the case of strong stability, this estimator is consistent when the potential outcomes satisfy generalized treatment effect additivity as defined in Definition~\ref{definition:additive-model}. The above discussion is summarized in the following theorem.

\begin{theorem}\label{corr:vaipw_est_weak}
For weakly stable designs (Definition~\ref{def:weak-design-stability}), under the sufficient condition~\eqref{eqn:extra_assumption}, the estimators $\widetilde{\sigma}_0^2$ and $\widetilde{\sigma}_1^2$ from~\eqref{eq:sigma0_aipw_weak} and~\eqref{eq:sigma1_aipw_weak} are consistent for $\Yvar_0$ and $\Yvar_1$, respectively. 
Furthermore, the variance estimator $\widehat{\vaipwweakhat}$ given by~\eqref{eqn:varaipw_est_weak} provides a conservative estimate of $\vaipwweak$, 
and is consistent when the potential outcomes satisfy generalized treatment effect additivity~\eqref{eqn:gen-additive-model}.
\end{theorem}
%
Refer to Section~\ref{sec:Proof-of-Corr:vaipw_est_weak} of the Supplementary Material for a proof of the theorem. 

\begin{remark} \label{rem:APIWunknown}
If for a strongly stable design the limiting value $\pstar$ is unknown or difficult to compute explicitly, substitution of the consistent estimator $\widehat{p}^\star$ defined in \eqref{eqn:phat} in place of $\pstar$ into \eqref{eqn:vaipw_liberal} will lead to an estimator of $\vaipwstrong$ with similar properties as in Theorem~\ref{corr:vaipw_est}. If for a weakly stable design the limiting values $\ponestar, \ptwostar$ and $\widetilde{p}$ are unknown or are difficult to compute explicitly, we can estimate them using~\eqref{eqn:weak_stab_consistent_ponetwo}. Substituting these estimators into \eqref{eqn:varaipw_est_weak}, a conservative estimator of $\vipwweak$ under weak stability is: 
\begin{align}
    \widehat{\vaipwweakhat} = \widetilde{\sigma}_0^2 \left(\frac{1}{N}\sum_{i=1}^N \frac{1}{1-p_i} - 1 \right) + \widetilde{\sigma}_0^2 \left(\frac{1}{N}\sum_{i=1}^N \frac{1}{p_i} - 1 \right) + 2 \widetilde{\sigma}_{0} \widetilde{\sigma}_{1},
\label{eqn:varaipw_est_weak_unknown}
\end{align}
where $\widetilde{\sigma}_0^2 = \displaystyle \frac{1}{\sum_{i=1}^N(1-p_i)} \sum_{i=1}^N (1-K_i)\big( Y_i - \widehat{Y}_{i-1}(0) \big)^2$ and $\widetilde{\sigma}_1^2 = \displaystyle \frac{1}{\sum_{i=1}^Np_i} \sum_{i=1}^N K_i\big( Y_i - \widehat{Y}_{i-1}(1) \big)^2.$
\end{remark}

\medskip
As in Section~\ref{sec:ipw_results}, for any target level $\alpha \in (0,1)$, 
asymptotically conservative confidence intervals for~$\ate$ may be constructed for $\snaipw$ 
using the variance estimators~\eqref{eqn:vaipw_liberal} and~\eqref{eqn:varaipw_est_weak} 
for strongly stable and weakly stable designs, respectively. 
If the limiting values of the probabilities are unknown, their counterparts suggested in 
Remark~\ref{rem:APIWunknown} may be used. 
All of these intervals are asymptotically conservative, but attain exact asymptotic coverage when the potential outcomes satisfy generalized treatment effect additivity~\eqref{eqn:gen-additive-model}.


\section{Some illustrative applications}
\label{sec:application}
In this section, we illustrate Theorems~\ref{Thm:ipw}-\ref{corr:vaipw_est_weak} through two popular adaptive designs: a strongly stable design, Wei's adaptive coin design~\cite{Wei1978-gd}, and a weakly stable design, Efron's biased coin design~\cite{Efron1971-yh}. We further complement the theoretical results with numerical simulations that demonstrate the validity of our approach.

\subsection{Wei's Adaptive Coin Design}\label{sec:wei}
We begin with Wei’s adaptive coin design~\cite{Wei1978-gd}, which reduces relative imbalance between treatment and control allocations. Formally, let $m_k$ and $n_k$ denote, respectively, the numbers of treatment and control units among the first $k$ subjects. Define the treatment-control imbalance up to the $k$th assignment as $D_k = m_k - n_k = 2 m_k - k$, and the corresponding normalized imbalance $R_k = D_k/k,$ which measures the average difference between the treatment and control groups up to stage $k$. The $i$th subject is assigned to treatment with probability  
\begin{align}
\label{eqn:wei-design}
p_{i} = f\left(R_{i-1}\right),
\end{align}  
where $f : [-1,1] \to [0,1]$ is a non-increasing function satisfying (i) $f(0) = \frac{1}{2}$ and (ii)
$f$ is continuous at zero. For the estimators \(\snipw\) and \(\snaipw\) to be well-defined under this design, it is necessary that the inclusion probabilities be bounded away from zero and one.  
If $f$ does not guarantee this property, we may enforce it by replacing $p_i$ in~\eqref{eqn:wei-design} with the clipped version  
\begin{align}
\label{eqn:truncated-p}
p_{i} = \min\big\{\max\{f(R_{i-1}), \delta\},\, 1-\delta\big\},
\end{align}  
for some fixed $\delta \in (0, \frac{1}{2}]$.  
This modification ensures $p_i \in [\delta,\, 1-\delta]$ for all $i \geq 1$. 

Intuitively, when the trial is in its early stages, the number of units in each group can differ substantially in relative terms; the design then shifts $p_i$ away from $\tfrac{1}{2}$ to favor the smaller group and reduce imbalance. As the sample size grows, any absolute difference in group sizes becomes small relative to the total number of units, causing $R_{i-1}$ to shrink and $p_i$ to converge to $\tfrac{1}{2}$. The following Lemma, which is a direct consequence of \cite[Theorem 1]{Wei1978-gd}, establishes strong stability of the truncated version of Wei's design~\eqref{eqn:truncated-p}, making Theorems~\ref{Thm:ipw},~\ref{corr:vipw_est},~\ref{Thm:aipw},~\ref{corr:vaipw_est} directly applicable.

\begin{lemma} \label{lemma:Wei_stability}
Wei's adaptive coin design~\eqref{eqn:truncated-p} is strongly stable in sense of Definition~\ref{def:design-stability}, with $\pstar = \frac{1}{2}$.
\end{lemma}
See Section~\ref{sec:Proof-of-Lemma:Wei_stability} of the Supplementary Material for the proof of Lemma~\ref{lemma:Wei_stability}. 


Substituting $\pstar = \tfrac{1}{2}$ into the expressions for $\vipwstrong$ and $\vaipwstrong$ in Theorems~\ref{Thm:ipw} and~\ref{Thm:aipw} yields the limiting variances of the IPW and AIPW estimators, $\snipw$ and $\snaipw$, respectively, under Wei’s design:
\begin{subequations}
\begin{align}
\vipwwei &= \moment_{0}+\moment_{1} + 2\crossmoment, 
\label{eqn:vipwwei} \\ 
\vaipwwei &= \Yvar_{0}+\Yvar_{1}+2\Ycov.
\label{eqn:vaipwwei}
\end{align}
\end{subequations}
Since $\pstar$ is known, it can be directly plugged into the variance estimators $\widehat{\vipwstronghat}$ in~\eqref{eqn:vipw_liberal} and $\widehat{\vaipwstronghat}$ in~\eqref{eqn:vaipw_liberal}, yielding the following conservative estimators for the IPW and AIPW variances:
\begin{align}\label{eqn:wei_var_est}
\widehat{\vipwwei} = (\widehat{m}_{0} + \widehat{m}_{1})^2, \qquad 
\widehat{\vaipwwei} = (\widehat{\sigma}_{0} + \widehat{\sigma}_{1})^2,
\end{align}
where $\widehat{m}^2_{\ell}$ and $\widehat{\sigma}^2_{\ell}$, $\ell \in \{0,1\}$, are as defined in~\eqref{eqn:mhats} and~\eqref{eqn:phat-aipw}, respectively. We can now use these variance estimators to construct conservative confidence intervals for $\ate$. 
Recall that the interval based on $\widehat{\vipwwei}$ attains exact asymptotic coverage when the potential outcomes satisfy additivity on the log scale~\eqref{eqn:log-additive-model}, whereas the interval based on $\widehat{\vaipwwei}$ attains exact asymptotic coverage under generalized treatment effect additivity as defined in~\eqref{eqn:gen-additive-model}.

Next, we evaluate the performances of the IPW and AIPW estimators under Wei’s adaptive coin design through simulation studies. 
We consider three types of potential outcome matrices: 
(a) A general, non-additive setting, in which the potential outcomes 
\((Y_i(0), Y_i(1))\) are drawn from a bivariate normal distribution with mean vector \((0,1)^\T\), variance vector \((1,1)^\T\) and correlation 0.3,
with support restricted to \([-3,3]\) to ensure bounded outcomes, 
(b) an additive setting satisfying \eqref{eqn:additive-model}, where the control potential outcomes are drawn from a normal distribution with mean~0 and variance~1, truncated to \([-3,3]\), and the treatment outcomes are defined by \(Y_i(1) = Y_i(0) + \tau\) with \(\tau = 10\); and 
(c) a log-additive setting satisfying \eqref{eqn:log-additive-model}, in which the control potential outcomes are drawn from a normal distribution with mean~10 and variance~1, truncated to \([7,13]\), and the treatment outcomes are defined by \(Y_i(1) = c\,Y_i(0)\) with \(c = 2\). Potential outcomes generated once are held fixed.
 
Treatment assignments \(\mathbf{K}\) are generated according to Wei’s sequential randomization scheme, with assignment probabilities 
\(p_i = f(R_{i-1}) = \left(1 - R_{i-1}\right)/2 = 1 - m_{i-1}/(i-1)\), 
where \(m_{i-1}\) denotes the number of units assigned to the treatment among the first $(i-1)$ subjects, and the truncation parameter is set to \(\delta = 0.01\). 
Each simulation involves \(N = 5000\) units and is repeated \( M= 20000\) times. 
For each replication, confidence intervals are constructed as proposed, and empirical coverage is evaluated across 20 nominal levels ranging from 0.75 to 0.99.

\begin{figure}[!ht]
\centering
  \begin{subfigure}{0.32\textwidth}
    \centering
    \includegraphics[width=\linewidth]{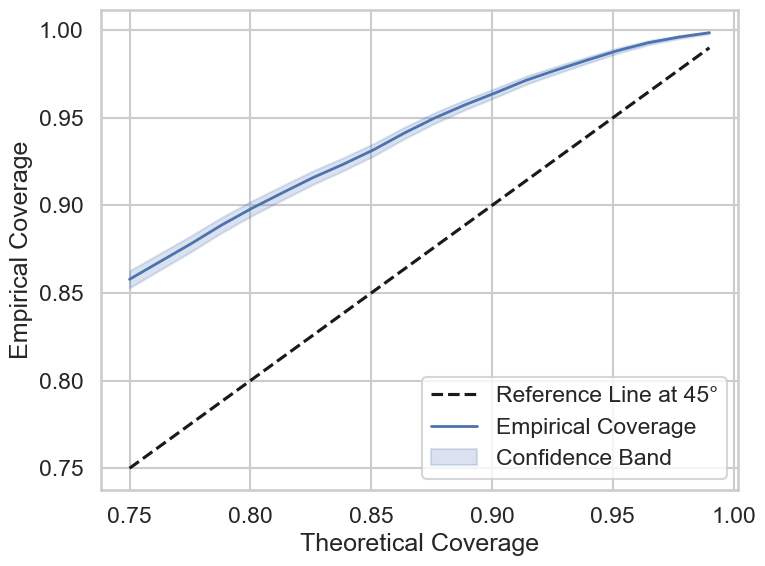}
    \caption{IPW estimator: Non-Additive}
    \label{fig:subfig1iia}
  \end{subfigure}%
  \begin{subfigure}{0.32\textwidth}
    \centering
    \includegraphics[width=\linewidth]{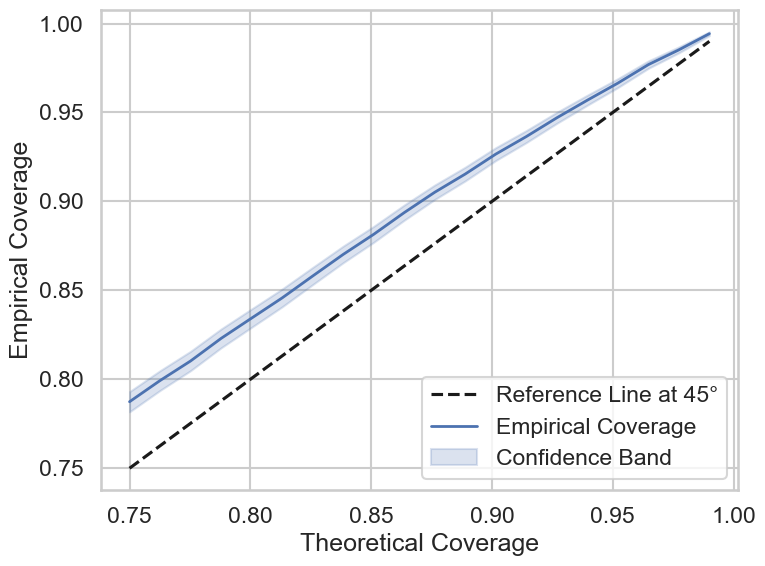}
    \caption{IPW estimator: Additive}
    \label{fig:subfig1ia}
  \end{subfigure}%
  \begin{subfigure}{0.32\textwidth}
    \centering
    \includegraphics[width=\linewidth]{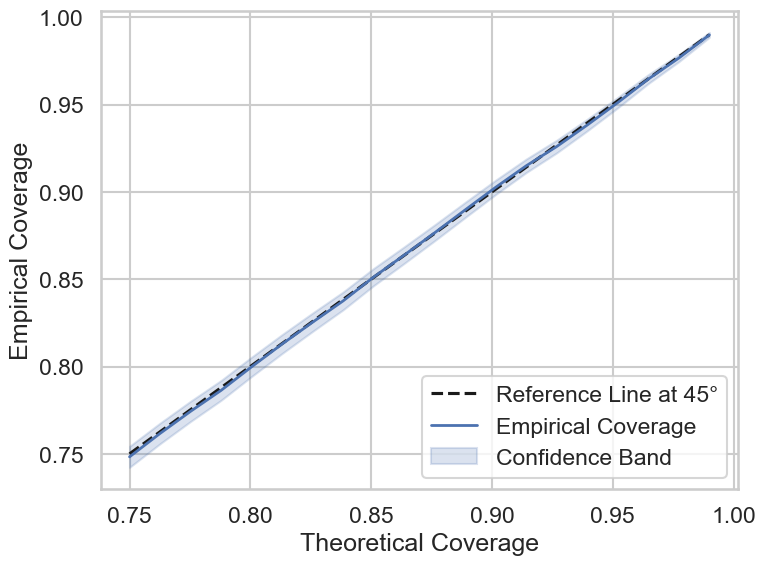}
    \caption{IPW estimator: Log-Additive}
    \label{fig:subfig1iva}
  \end{subfigure} \\[1ex]
  \begin{subfigure}{0.32\textwidth}
    \centering
    \includegraphics[width=\linewidth]{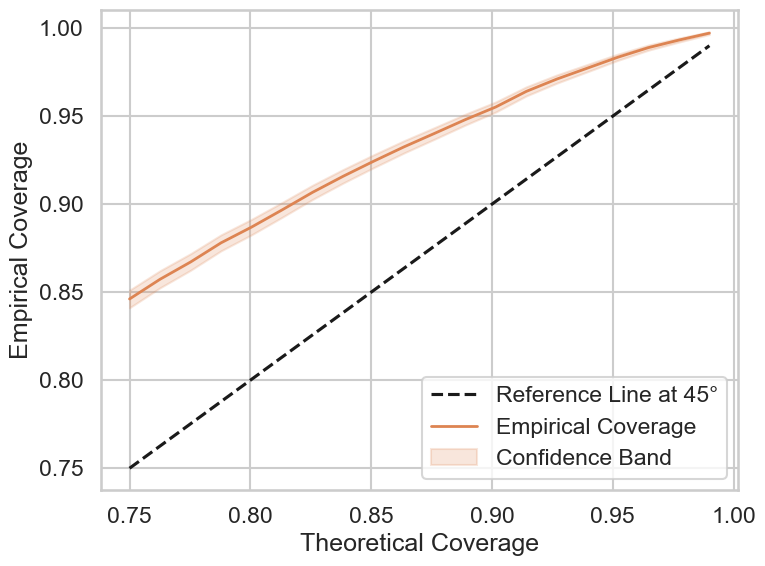}
    \caption{AIPW estimator: Non-Additive}
    \label{fig:subfig1iib}
  \end{subfigure} %
  \begin{subfigure}{0.32\textwidth}
    \centering
    \includegraphics[width=\linewidth]{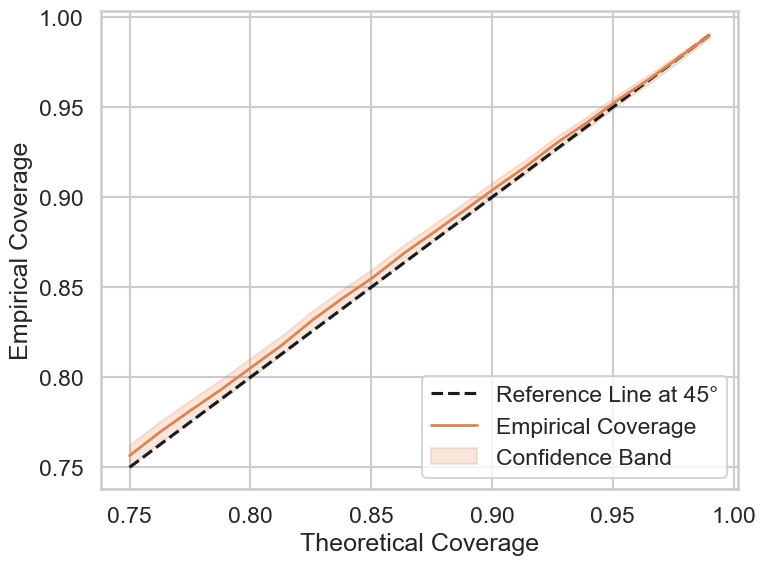}
    \caption{AIPW estimator: Additive}
    \label{fig:subfig1ib}
  \end{subfigure} %
  \begin{subfigure}{0.32\textwidth}
    \centering
    \includegraphics[width=\linewidth]{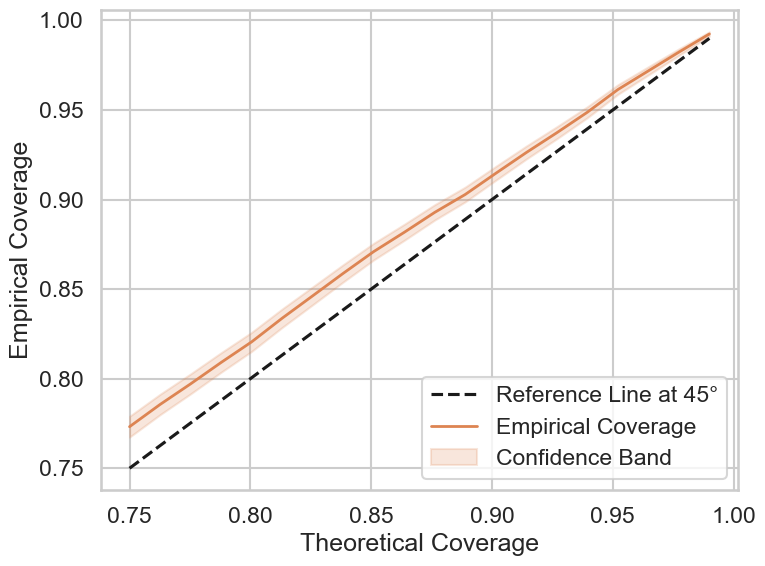}
    \caption{AIPW estimator: Log-Additive}
    \label{fig:subfig1ivb}
  \end{subfigure} %
  \caption{Comparison of the theoretical and empirical coverages for Wei's design.}
  \label{fig:weiplots}
\end{figure}

Figure~\ref{fig:weiplots} reports the empirical coverage of confidence intervals based on the IPW and AIPW estimators. 
In both cases, the intervals exhibit reliable coverage of the true $\ate$. 
For the non-additive setup, both intervals remain conservative, whereas under additivity and log-additivity, the empirical coverage approaches the nominal levels. 
In particular, the AIPW estimator performs better in the non-additive setting, yielding coverage closer to the nominal levels than the IPW estimator. 
The IPW estimator attains nearly exact coverage under the log-additive setup, while the AIPW estimator achieves nearly exact coverage under additivity and remains close to nominal levels under log-additivity. 
These results align with the theoretical guarantees of variance estimator consistency established in Theorems~\ref{corr:vipw_est} and~\ref{corr:vaipw_est}, and overall demonstrate the superior stability of the AIPW estimator.

Figure~\ref{fig:weiciplots} displays the average confidence interval lengths for the true parameter $\ate$ under the IPW and AIPW estimators. 
Each interval length is computed as \(2 \, z_{1-\alpha/2} \, \sqrt{\widehat{V}/N},\) 
where \(z_{1-\alpha/2}\) is the standard normal quantile corresponding to the nominal level \(\alpha\), and \(\widehat{V}\) denotes the estimated variance. 
For the IPW and AIPW estimators, this corresponds to \(\widehat{\vipwwei}\) and \(\widehat{\vaipwwei}\), respectively, as defined in equation~\eqref{eqn:wei_var_est}. 
Across all confidence levels and data-generating mechanisms, the AIPW estimator produces substantially shorter intervals than the IPW estimator. Together with the coverage results established in Figure~\ref{fig:weiplots}, these results highlight the overall greater efficiency and stability of the AIPW estimator. Specifically, while the IPW estimator achieves valid coverage under log-additive setup, the AIPW estimator performs remarkably better in both additive and non-additive setups, providing coverage levels closer to the nominal values along with consistently shorter confidence intervals. In general, these results emphasize that although both estimators attain reliable coverage, the AIPW estimator achieves this with noticeably tighter intervals, making it generally more efficient and preferable in practical applications. 

\begin{figure}[ht]
\centering
  \begin{subfigure}{0.33\textwidth}
    \centering
    \includegraphics[width=\linewidth]{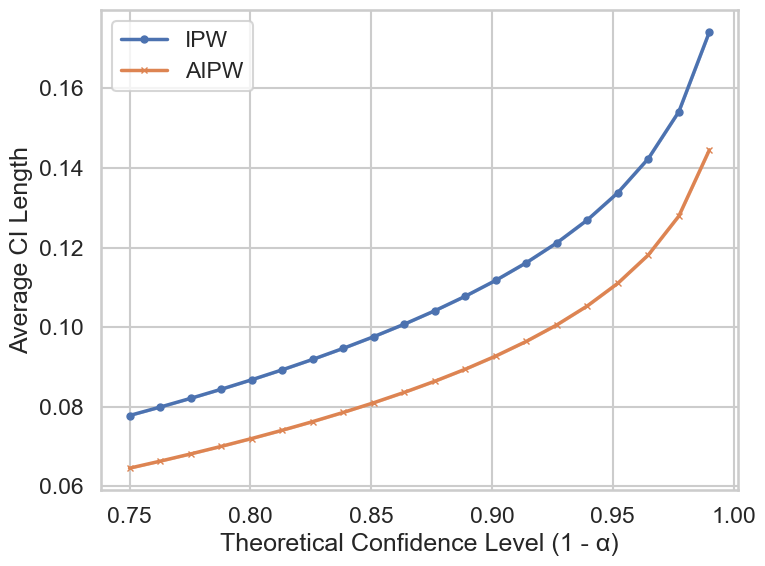}
    \caption{Non-Additive case}
    \label{fig:subfig1iiia}
  \end{subfigure}%
  \hfill
  \begin{subfigure}{0.33\textwidth}
    \centering
    \includegraphics[width=\linewidth]{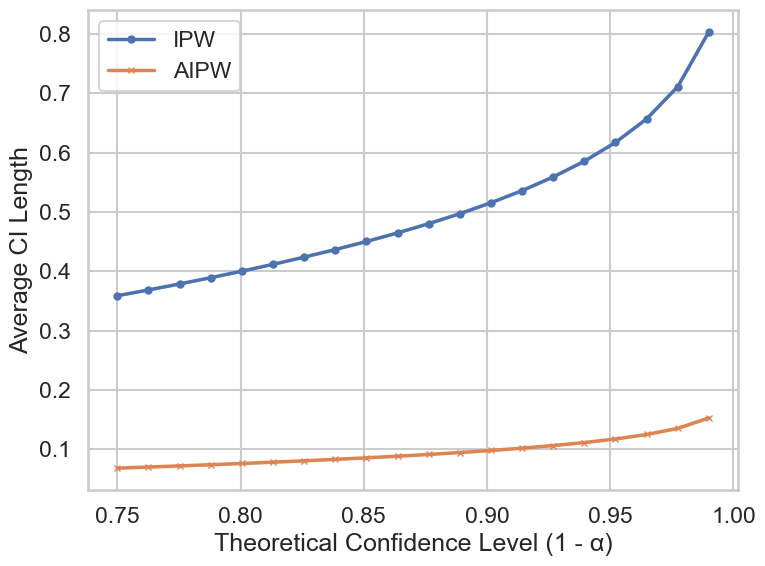}
    \caption{Additive case}
    \label{fig:subfig1iiib}
  \end{subfigure}%
  \hfill
  \begin{subfigure}{0.33\textwidth}
    \centering
    \includegraphics[width=\linewidth]{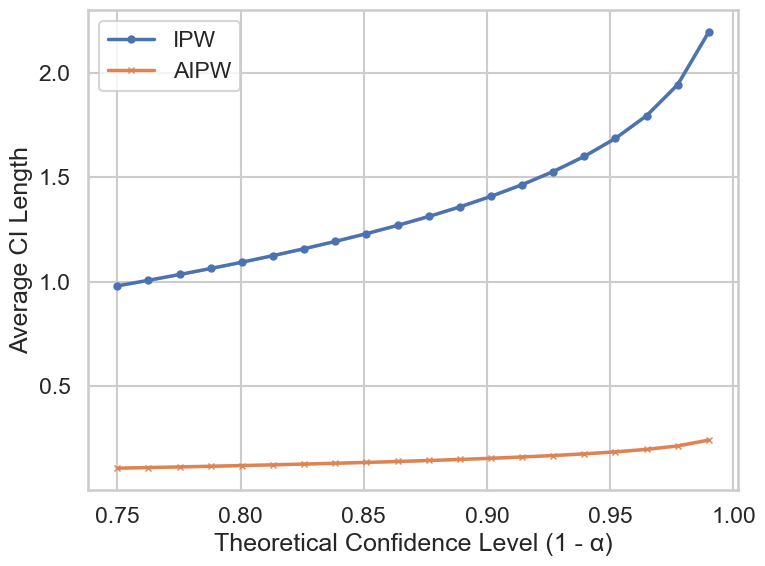}
    \caption{Log-Additive case}
    \label{fig:subfig1va}
  \end{subfigure}%
  \caption{Comparison of the average lengths of confidence intervals for Wei's design.}
  \label{fig:weiciplots}
\end{figure} 

Table~\ref{tab:bias_vr_wei} reports additional diagnostic metrics across sample sizes. Specifically, we report the Monte Carlo bias, computed as the difference between the average estimator across $M=2000$ simulation replications and the true treatment effect, and the variance ratio, defined as the ratio of the empirical Monte Carlo variance to the average estimated variance across replications. The variance ratio evaluates the calibration of the proposed variance estimators, with values close to one indicating accurate variance estimation.

While the coverage and confidence interval length plots are based on $M=20000$ Monte Carlo replications to ensure stable estimation of tail probabilities, the bias and variance ratio diagnostics in Table~\ref{tab:bias_vr_wei} are computed using $M=2000$ replications, which are sufficient for accurate estimation of first- and second-moment quantities. While coverage and confidence interval length results for $N=5000$ are presented graphically, Table~\ref{tab:bias_vr_wei} summarizes Monte Carlo bias and variance ratio diagnostics for $N=500$ and $N=10000$.

Several patterns emerge. First, the Monte Carlo bias is negligible across sample sizes, confirming the unbiasedness properties established theoretically. Second, the variance ratios exhibit mild finite sample deviations at $N=500$, particularly under the log-additive design, but generally move closer to one as $N$ increases to $10000$. This behavior is consistent with the theoretical results and illustrates convergence towards the theoretical large sample regime.

\begin{table}[H]
\centering
\caption{Finite sample bias and variance ratio for $N=500$ and $N=10000$ under Wei's design.}
\label{tab:bias_vr_wei}
\small
\setlength{\tabcolsep}{4pt}
\begin{tabular}{ll
                S[table-format=+1.5] S[table-format=+1.5]
                S[table-format=2.4]  S[table-format=2.4]}
\toprule
 & 
& \multicolumn{2}{c}{Bias}
& \multicolumn{2}{c}{Variance Ratio} \\
\cmidrule(lr){3-4}\cmidrule(lr){5-6}
Model & Estimator
& {N=500} & {N=10000}
& {N=500} & {N=10000} \\
\midrule

Additive &  IPW  & -0.01385  & 0.00211   & 0.9317 & 0.8600 \\
 & AIPW &  0.005922 & 0.0005493 & 1.1050 & 0.9046 \\
\addlinespace[4pt]

Non-additive & IPW  &  0.001238 & 0.001204  & 0.9379 & 0.8553 \\
& AIPW &  0.002993 & 0.001014  & 1.2190 & 0.9885 \\

\addlinespace[4pt]

Log-additive & IPW  & -0.04029  & 0.005665  & 1.0670 & 1.0140 \\
\quad & AIPW &  0.005526 & 0.0005591 & 0.9561 & 1.0090 \\

\bottomrule
\end{tabular}
\end{table}

\subsection{Efron's Biased Coin Design}\label{sec:efron}
Next, we consider the biased coin design introduced by~\cite{Efron1971-yh} which enforces another form of approximate balance between the number of allocations in the treatment and control groups. As in the previous section, let \(D_k = m_k - n_k = 2m_k - k\) denote the imbalance between the treatment and control groups after the assignment of the \(k\)th unit, 
where \(m_k\) and \(n_k\) denote the numbers of treatment and control assignments, respectively. Under Efron's biased coin design, the $i$th unit is assigned to treatment with probability:
\begin{align}
\label{eqn:Efron-design}
p_{i} =
\begin{cases}
\eta & \text{if } D_{i-1} < 0\\
\frac{1}{2} & \text{if } D_{i-1} = 0 \\
1 - \eta & \text{if } D_{i-1} > 0
\end{cases},
\end{align} 
where $\eta \in \big[\frac{1}{2} , 1)$ controls the strength of the bias toward balance. In words, a larger value of $\eta$ forces faster correction of imbalance. It is worth noting that $p_i$ takes the values $\eta$ and $1 - \eta$ infinitely often, and thus Efron’s biased coin design is not strongly stable in the sense of Definition~\ref{def:design-stability}. However, the following lemma establishes weak stability of the design, making Theorems~\ref{Thm:ipw}, \ref{corr:vipw_est_weak}, \ref{Thm:aipw}, and~\ref{corr:vaipw_est_weak} applicable.
\begin{lemma} \label{lemma:efron_as_half}
Efron's biased coin design~\eqref{eqn:Efron-design} is weakly stable in sense of Definition~\ref{def:weak-design-stability}, with $\ponestar = \frac{4\eta^2(1-\eta)}{1-4\eta+12\eta^2-8\eta^3}$ and $\ptwostar = \frac{1-4\eta+8\eta^2-4\eta^3}{1-4\eta+12\eta^2-8\eta^3}$.  Moreover,
$N^{-1} \sum_{i=1}^N p_i \stackrel{\mathbb{P}}{\rightarrow} \frac{1}{2}$.
\end{lemma}

\begin{remark}
The proof of Lemma~\ref{lemma:efron_as_half} proceeds by studying the treatment–control imbalance sequence $\{D_k\}_{k \geq 1}$. We first establish that $\{D_k\}_{k \geq 1}$ forms an irreducible and positively recurrent Markov chain by applying Foster’s Theorem~\cite{Foster1953-av}. The resulting positive recurrence and irreducibility ensure the existence of a unique stationary distribution, which, together with the mean ergodic theorem, facilitates characterization of the limiting behavior of long-run averages of functions of the assignment probabilities. Details of the proof are provided in Section~\ref{sec:Proof-of-Lemma:efron_as_half} of the Supplementary Material.
\end{remark}

Substituting the values of $\ponestar$ and $\ptwostar$ from Lemma~\ref{lemma:efron_as_half} into Theorems~\ref{Thm:ipw} and~\ref{Thm:aipw} yields the limiting variances of the IPW and AIPW estimators, $\snipw$ and $\snaipw$, under Efron’s design:
\begin{subequations}
\begin{align}
\vipwefron &= \left(\moment_{0}+\moment_{1}\right)\frac{1-4\eta+8\eta^2-4\eta^3}{4\eta^2(1-\eta)} + 2\crossmoment,
\label{eqn:vipwefron} \\
\vaipwefron &= \left(\Yvar_{0}+\Yvar_{1}\right)\frac{1-4\eta+8\eta^2-4\eta^3}{4\eta^2(1-\eta)} + 2\Ycov.
\label{eqn:vaipwefron}
\end{align}
\end{subequations}

Furthermore, as shown in Lemma~\ref{lemma:efron_as_half}, under this design $N^{-1}\sum_{i=1}^N p_i \stackrel{\mathbb{P}}{\rightarrow} \tfrac{1}{2}$, thereby satisfying the sufficient condition in~\eqref{eqn:extra_assumption}. This allows for consistent estimation of $\moment_0$, $\moment_1$, $\Yvar_0$, and $\Yvar_1$. Substituting these into~\eqref{eqn:varipw_est_weak} and~\eqref{eqn:varaipw_est_weak} yields conservative estimators of $\vipwefron$~\eqref{eqn:vipwefron} and $\vaipwefron$~\eqref{eqn:vaipwefron}:


\begin{subequations}
\begin{align}
\widehat{\vipwefron} &= \left(\widehat{m}^2_{0}+\widehat{m}^2_{1}\right)\frac{1-4\eta+8\eta^2-4\eta^3}{4\eta^2(1-\eta)} + 2\widehat{m}_{0}\widehat{m}_{1}, \label{eqn:vipwefron_hat} \\
\widehat{\vaipwefron} &= \left(\widehat{\sigma}^2_{0}+\widehat{\sigma}^2_{1}\right)\frac{1-4\eta+8\eta^2-4\eta^3}{4\eta^2(1-\eta)} + 2\widehat{\sigma}_{0}\widehat{\sigma}_{1}, 
\label{eqn:vaipwefron_hat}
\end{align}
\end{subequations} As in Section~\ref{sec:wei}, the variance estimators $\widehat{\vipwefron}$ and $\widehat{\vaipwefron}$ can be used to construct conservative confidence intervals for $\ate$. 
The interval based on $\widehat{\vipwefron}$ achieves exact asymptotic coverage when the potential outcomes satisfy additivity on the log scale~\eqref{eqn:log-additive-model}, 
while the interval based on $\widehat{\vaipwefron}$ attains exact coverage under generalized treatment effect additivity, as defined in~\eqref{eqn:gen-additive-model}.

A complete assessment of the coverage of the confidence intervals constructed using the IPW and AIPW estimators under Efron's design is provided in Section~\ref{sec:simulation_for_efron} of the Supplementary Material. The simulations use the same data-generating procedures and parameter settings as in Section~\ref{sec:wei}, 
except that treatment assignments now follow Efron’s biased coin design~\eqref{eqn:Efron-design} rather than Wei’s design. The biased-coin parameter is fixed at $\eta = 0.7$ in all simulations.



\begin{remark}
Both set of simulations conducted in Sections \ref{sec:wei} and \ref{sec:efron} reveal a common interesting pattern: when the nominal confidence level is high, the variance estimator appears less conservative, with empirical coverage closer to nominal. This phenomenon can be explained from the asymptotic theory. See Section~\ref{sec:theoretical_explanation_of_observed_coverage_patterns} of the Supplementary Material for theoretical justification.
\end{remark}





\subsection{Possible extensions to more general designs in adaptive clinical trial literature and assessment of stability} \label{sec:clinical_trial}

As discussed in \cite{Adaptive_clinical1} and \cite{Adaptive_clinical2}, Wei's and Efron's designs can be considered as special cases of more general classes of adaptive designs used in clinical trial literature, in which the assignment probability $p_i$ is a known function of the form $g \left(m_{i-1}/(i-1), \widehat{p}^\star_{i-1} \right)$, where as earlier, $m_k$ denotes the number of treatment units among the first $k$ subjects, and $\widehat{p}_{i-1}^\star$ is an estimator of the unknown target proportion $p^\star$ after stage $(i-1)$. For designs in which $g(x,y)$ only depends on $x$ (e.g., Wei's design with $g(x,y) = 1-x$), it is always possible to simulate sequences of assignment probabilities $p_i$ and plot them against step $i$ to see whether stability is attained (refer to Section~\ref{sec:practical_assessment_of_stability} of the Supplementary Material for additional details). For designs where the target proportion $p^\star$ is unknown, but a simple estimator of the form~\eqref{eqn:phat} is used, such a simulation-based assessment is also possible. However, when the true target proportion is a known function of the unknown potential outcomes (e.g., in Neymanian allocation as in~\cite{Ravi2024}) and the estimated proportion is a function of the observed outcomes, imposing additional conditions on the potential outcomes may be necessary to conduct such simulation studies. This requires further investigation and should be an interesting area for further exploration.

\section{Discussion}
\label{sec:discussion}

We have developed a general theoretical framework for conducting inference on average treatment effects in settings where treatment assignment is sequentially adaptive within a finite-population. This framework unifies and extends existing results by accommodating a broad class of adaptive randomization schemes, where assignment probabilities may evolve over time based on past outcomes. Within this setup, we establish central limit theorems for both inverse probability weighted (IPW) and augmented IPW (AIPW) estimators under strong and weak design stability conditions. Although the limiting distributions feature explicit expressions for the asymptotic variances, the fundamental problem of causal inference - not being able to observe the two potential outcomes for each unit - leads to challenges in their estimation. We propose conservative variance estimators that are consistent under different forms of treatment effect homogeneity.

To demonstrate the applicability of our framework, we analyze Wei’s adaptive coin design and Efron’s biased coin design, two classical examples in sequential experimentation. These applications reveal how the general theory accommodates designs that deviate from strong stability (e.g., Efron's design), thereby illustrating its flexibility and robustness.

From a practical standpoint, our findings provide reassurance that adaptive treatment assignment mechanisms—increasingly popular in modern experimental and clinical trial settings—can be used within a finite population framework without imposing any model on the potential outcome. The research opens up several new research possibilities. Extending the framework to covariate-adaptive designs where assignments depend explicitly on pre-measured covariates \cite[e.g.,][]{FSM1979, morgan2012}, would broaden the applicability of the theory. Adaptive treatment assignment mechanisms also provide a natural solution to finding optimal designs in a finite-population setting, e.g., \cite{Ravi2024} and the results presented in this paper can provide an inferential framework for such adaptive designs.



\section*{Acknowledgement}
We thank the two reviewers and the Associate Editor for insightful and helpful comments. The research of Tirthankar Dasgupta was partially funded by National Science Foundation Grant SES 2217522. The research of Koulik Khamaru was partially funded by National Science Foundation Grant DMS-2311304.

\printbibliography

\newpage

\renewcommand{\thesection}{S\arabic{section}}
\renewcommand{\thesection}{S\arabic{section}}
\renewcommand{\thesubsection}{S\arabic{section}.\arabic{subsection}}

\section{Simulation results for Efron’s biased coin design~\cite{Efron1971-yh}}
\label{sec:simulation_for_efron}

As mentioned earlier, the simulations use the same data-generating procedures and parameter settings as in Section~5.1 of the main paper, 
except that treatment assignments now follow Efron’s biased coin design~\cite{Efron1971-yh} rather than Wei’s design~\cite{Wei1978-gd}. The biased-coin parameter is fixed at $\eta = 0.7$ in all simulations.

\begin{figure}[!ht]
\centering
  \begin{subfigure}{0.32\textwidth}
    \centering
    \includegraphics[width=\linewidth]{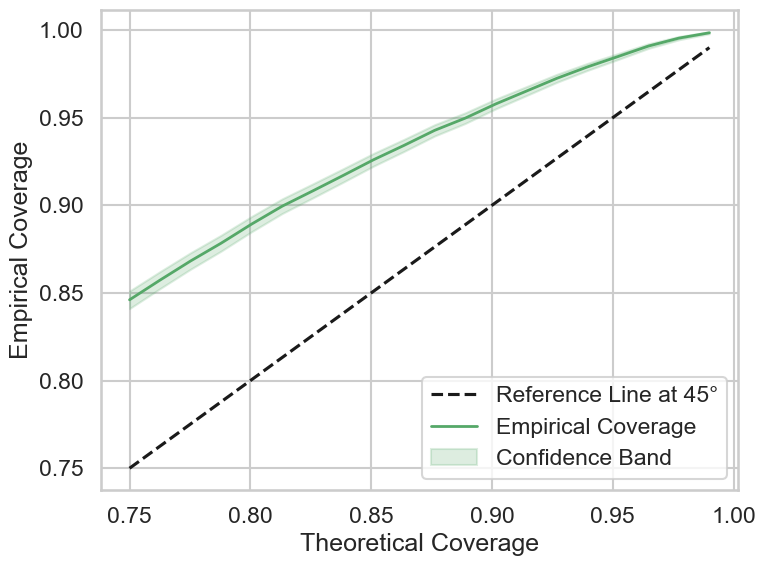}
    \caption{IPW estimator: Non-Additive}
    \label{fig:subfig1iia}
  \end{subfigure}%
  \begin{subfigure}{0.32\textwidth}
    \centering
    \includegraphics[width=\linewidth]{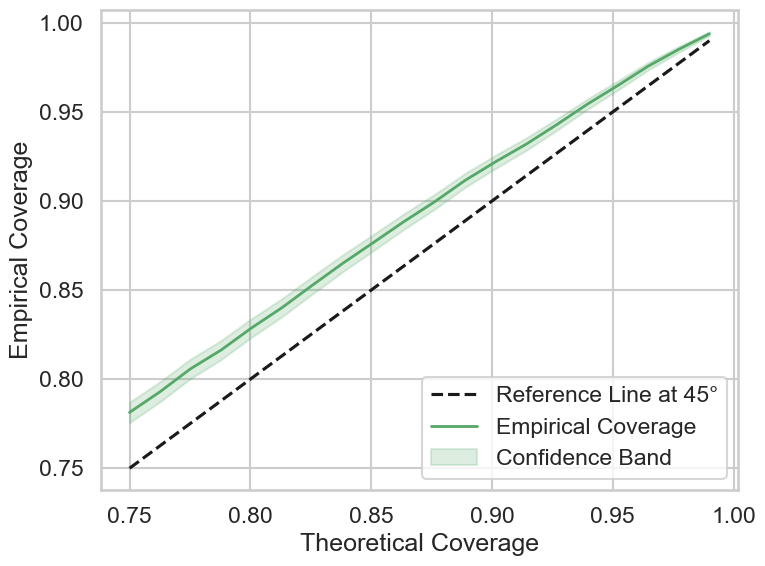}
    \caption{IPW estimator: Additive}
    \label{fig:subfig1ia}
  \end{subfigure}%
  \begin{subfigure}{0.32\textwidth}
    \centering
    \includegraphics[width=\linewidth]{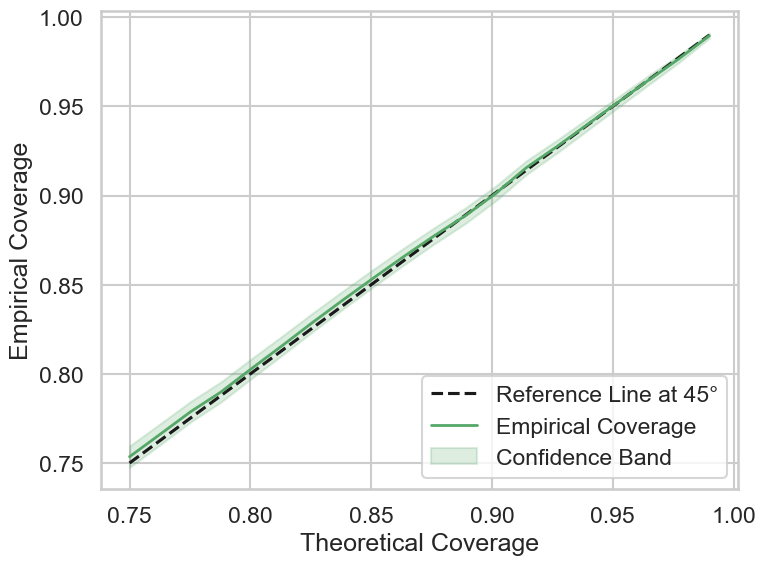}
    \caption{IPW estimator: Log-Additive}
    \label{fig:subfig1iva}
  \end{subfigure} \\[1ex]
  \begin{subfigure}{0.32\textwidth}
    \centering
    \includegraphics[width=\linewidth]{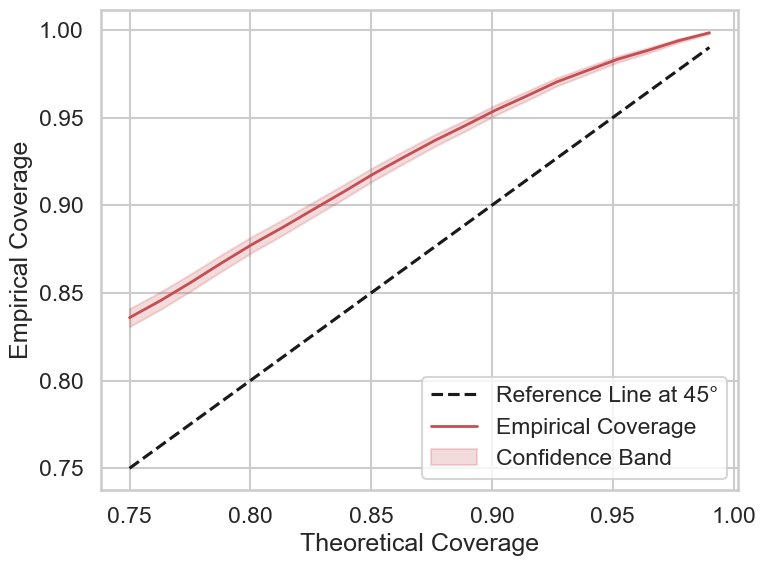}
    \caption{AIPW estimator: Non-Additive}
    \label{fig:subfig1iib}
  \end{subfigure} %
  \begin{subfigure}{0.32\textwidth}
    \centering
    \includegraphics[width=\linewidth]{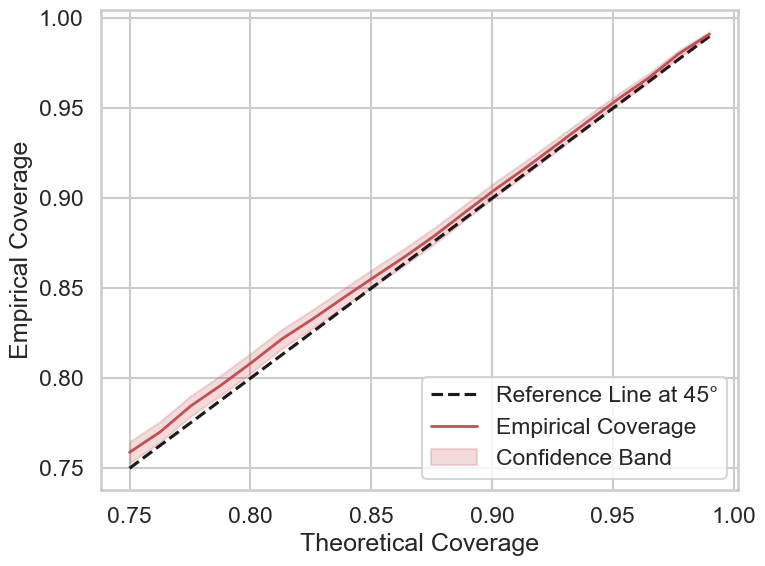}
    \caption{AIPW estimator: Additive}
    \label{fig:subfig1ib}
  \end{subfigure} %
  \begin{subfigure}{0.32\textwidth}
    \centering
    \includegraphics[width=\linewidth]{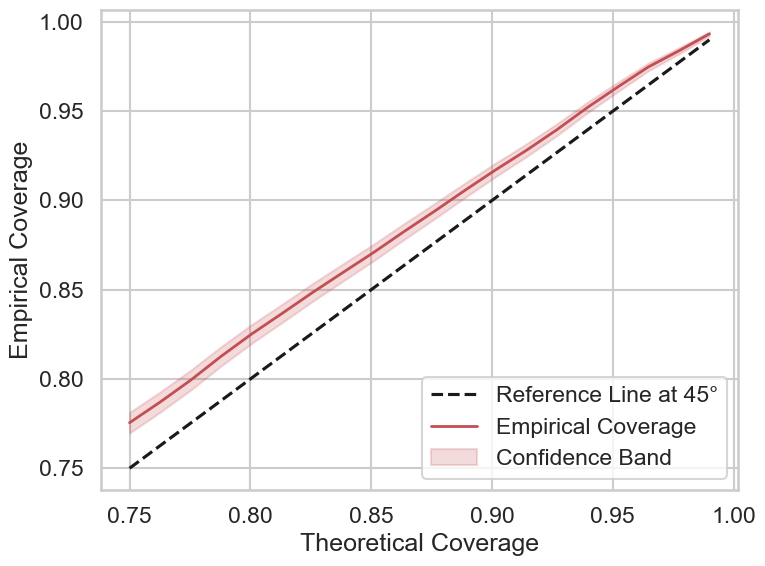}
    \caption{AIPW estimator: Log-Additive}
    \label{fig:subfig1ivb}
  \end{subfigure} %
  \caption{Comparison of the theoretical and empirical coverages for Efron's design.}
  \label{fig:efronplots}
\end{figure}

\begin{figure}[!ht]
\centering
  \begin{subfigure}{0.32\textwidth}
    \centering
    \includegraphics[width=\linewidth]{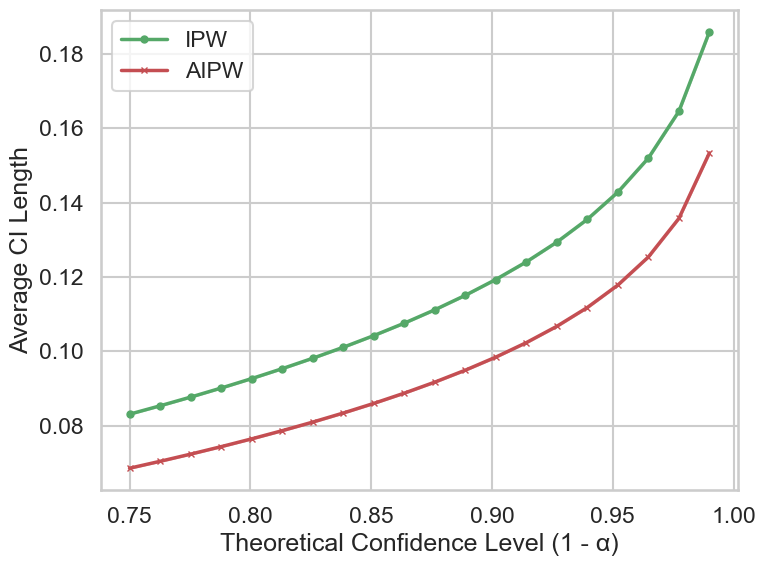}
    \caption{Non-Additive case}
    \label{fig:subfig2iiia}
  \end{subfigure}%
  \hfill
  \begin{subfigure}{0.32\textwidth}
    \centering
    \includegraphics[width=\linewidth]{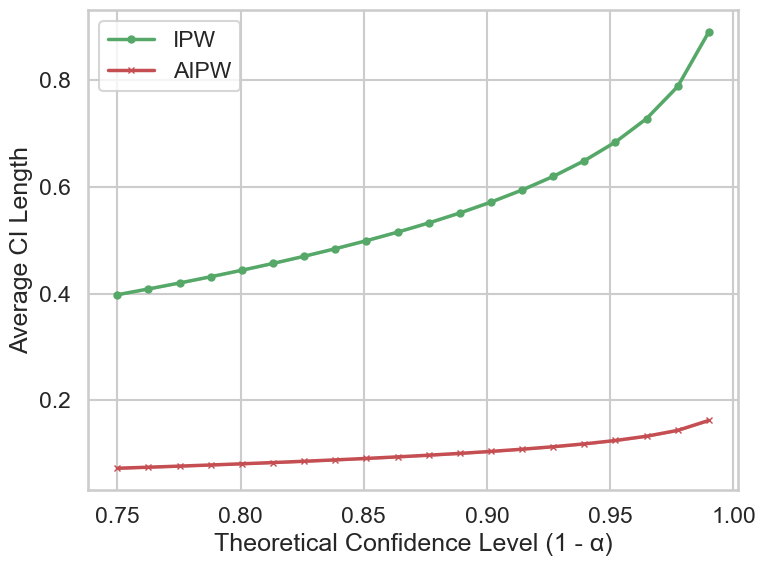}
    \caption{Additive case}
    \label{fig:subfig2iiib}
  \end{subfigure}%
  \hfill
  \begin{subfigure}{0.32\textwidth}
    \centering
    \includegraphics[width=\linewidth]{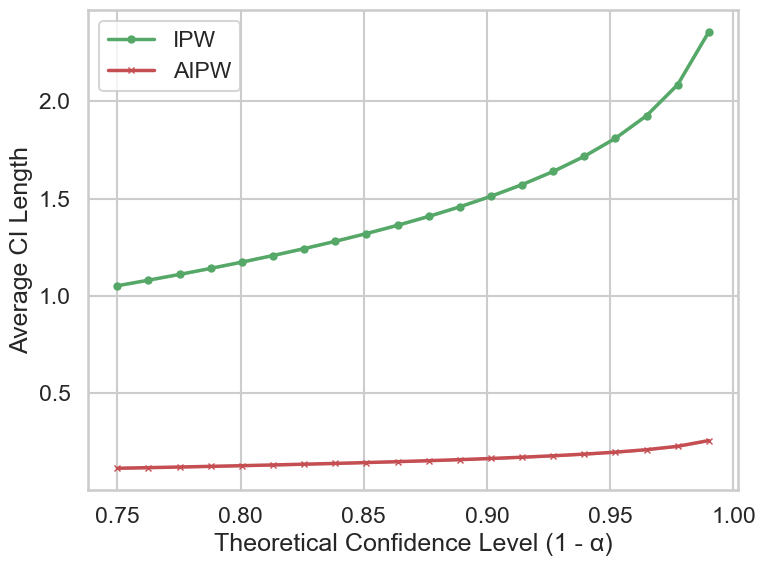}
    \caption{Log-Additive case}
    \label{fig:subfig2va}
  \end{subfigure}%
  \caption{Comparison of the average lengths of confidence intervals for Efron's design.}
  \label{fig:efronciplots}
\end{figure}

Figure~\ref{fig:efronplots} shows the empirical coverage of confidence intervals for the IPW and AIPW estimators, 
while Figure~\ref{fig:efronciplots} reports the corresponding average interval lengths. 
The lengths are computed as for Wei’s design, using $\widehat{\vipwefron}$ for IPW and $\widehat{\vaipwefron}$ for AIPW, 
as specified in equations~\eqref{eqn:vipwefron_hat} and~\eqref{eqn:vaipwefron_hat} of the main paper, respectively. Table \ref{tab:bias_vr_ef} summarizes Monte Carlo bias and variance ration diagnostics for $N=500$ and $N=1000$. The results for Efron’s biased coin design are consistent with those observed under Wei’s adaptive coin design and discussed in Section 5 of the main manuscript. 

\begin{table}[ht]
\centering
\caption{Finite sample bias and variance ratio for $N=500$ and $N=10000$ under Efron's design.}
\label{tab:bias_vr_ef}
\small
\setlength{\tabcolsep}{4pt}
\begin{tabular}{lll
                S[table-format=+1.5] S[table-format=+1.5]
                S[table-format=2.4]  S[table-format=2.4]}
\toprule
& & 
& \multicolumn{2}{c}{Bias}
& \multicolumn{2}{c}{Variance Ratio} \\
\cmidrule(lr){4-5}\cmidrule(lr){6-7}
Model & Design & Est
& {N=500} & {N=10000}
& {N=500} & {N=10000} \\
\midrule

\multicolumn{7}{l}{\textbf{Additive}}\\
\addlinespace[2pt]
\quad & Efron & IPW  &  0.006335 & -0.0002903& 0.8890 & 0.8387 \\
\quad & Efron & AIPW &  0.0003965& 0.0002315 & 0.9238 & 0.9827 \\
\addlinespace[4pt]

\multicolumn{7}{l}{\textbf{Non-additive}}\\
\addlinespace[2pt]
\quad & Efron & IPW  & -0.0007978& 0.0007354 & 0.8745 & 0.8672 \\
\quad & Efron & AIPW & -0.01473  & -0.0007568& 1.0290 & 0.9555 \\
\addlinespace[4pt]

\multicolumn{7}{l}{\textbf{Log-additive}}\\
\addlinespace[2pt]
\quad & Efron & IPW  &  0.02869  & -0.004243 & 0.9929 & 0.9538 \\
\quad & Efron & AIPW & -0.0004171& 0.0002339 & 0.7379 & 1.0280 \\
\bottomrule
\end{tabular}
\end{table}

\clearpage

\section{Proofs of Theorems}
In this section, we collect the proofs of our main Theorems~\ref{Thm:ipw}-\ref{corr:vaipw_est_weak}. We begin by recalling the IPW and AIPW estimators introduced in Section~\ref{sec:problem_desc} of the main paper. 
Before proceeding to the proofs, observe that when $K_i = 1$ we have $Y_i = \Yione$, and when $K_i = 0$ we have $Y_i = \Yizero$. 
Consequently, $K_iY_i = K_i\Yione$ and $(1-K_i)Y_i = (1-K_i)\Yizero$. 
Thus, the estimators simplify to  
\begin{subequations}
    \begin{align}
        \label{eqn:IPW1}
        \snipw &= \frac{1}{N} \sum_{i=1}^N \left\{\frac{K_i\Yione}{p_i} -  \frac{(1 - K_i)\Yizero}{1 - p_i}\right\}, \\
        \label{eq:SNAIPW1}
        \snaipw &= \frac{1}{N} \sum_{i=1}^N \left[\left\{  \frac{K_i\left(\Yione - \widehat{Y}_{i-1}(1)\right)}{p_i} + \widehat{Y}_{i-1}(1) \right\}
        - \left\{\frac{(1 - K_i)\left(\Yizero - \widehat{Y}_{i-1}(0)\right)}{1 - p_i} + \widehat{Y}_{i-1}(0) \right\}\right],
    \end{align}
\end{subequations}
where $\widehat{Y}_{1}(0) = \widehat{Y}_{1}(1) = 0$, and for $i \geq 2$
\begin{align*}
   \widehat{Y}_{i-1}(0) = \frac{1}{i-1} \sum_{j = 1}^{i - 1} \frac{(1 - K_j)Y_j(0) }{1 - p_j}, \qquad \text{and} \qquad \widehat{Y}_{i-1}(1) = \frac{1}{i-1} \sum_{j = 1}^{i - 1} \frac{K_jY_j(1) }{p_j}.
\end{align*}
In what follows, we work primarily with the representations~\eqref{eqn:IPW1} and~\eqref{eq:SNAIPW1}.

\subsection{Proof of Theorem~\ref{Thm:ipw} (CLT for the IPW estimator)}
\label{sec:Proof-of-Thm:ipw}
We begin by expressing the centered and scaled IPW  estimator $\sqrt{N}\left(\snipw - \ate\right) $ as a sum of a martingale difference sequence. 
This allows us to prove the central limit theorem for $\snipw$ via an application of the martingale central limit theorem \cite[Chapter 3]{Hall1980-ny}.  
\begin{align*}
\sqrt{N}\left(\snipw - \ate\right) 
&= \sum_{i=1}^N \frac{K_i-p_i}{\sqrt{N}}  \left( \frac{\Yizero}{1 - p_i} + \frac{\Yione}{p_i}\right) \;=\; \sum_{i=1}^N \xi_i,
\end{align*}
where $\xi_i = \tfrac{K_i-p_i}{\sqrt{N}}  \left(  \tfrac{\Yizero}{1 - p_i} + \tfrac{\Yione}{p_i}\right).$
Now, 
\begin{align*}
\E{\left[\xi_i \,\middle|\, \sigmafield\right]} = \frac{1}{\sqrt{N}}\left( \frac{\Yizero}{1 - p_i} + \frac{\Yione}{p_i}\right) \E{[K_i-p_i | \sigmafield]}  = 0,
\end{align*}
implying $\{\xi_i\}_{i=1}^N$ are terms of a martingale difference sequence, and that $\snipw$ is an unbiased estimator for $\ate$. Next, we compute the total conditional variance:   
\begin{align*}
\sum_{i=1}^N\E{\left[\xi_i^2 \,\middle|\, \sigmafield\right]} &= \sum_{i=1}^N\frac{1}{N}\left(\frac{\Yizero}{1-p_i} + \frac{\Yione}{p_i}\right)^2\E{\left[(K_i - p_i)^2 \,\middle|\, \sigmafield\right]} \\ 
&= \frac{1}{N}\sum_{i=1}^N \frac{p_i}{1-p_i}\Yizero^2 + \frac{1}{N}\sum_{i=1}^N \frac{1-p_i}{p_i}\Yione^2 + \frac{2}{N}\sum_{i=1}^N \Yizero\Yione.    
\end{align*}
In order to invoke the martingale central limit theorem \cite[Chapter 3]{Hall1980-ny}, we need to ensure that the total conditional variance converges in probability  to a constant and that the Lindeberg condition is satisfied. 

\vspace{10pt}
\noindent 
\underline{\textbf{For Strongly Stable Design:}} Under the assumption of strong design stability, we have $p_i \stackrel{\mathbb{P}}{\rightarrow} \pstar.$
Invoking the continuous mapping theorem in conjunction with Assumption~\ref{assn:IPW}(a) yields $\tfrac{p_i}{1-p_i} \stackrel{\mathbb{P}}{\rightarrow} \tfrac{\pstar}{1-\pstar}.$
Therefore, Assumption~\ref{assn:IPW} along with Lemma~\ref{lemma:aibi} implies 
\begin{align*}
\frac{1}{N}\sum_{i=1}^N \frac{p_i}{1-p_i}\Yizero^2 \stackrel{\mathbb{P}}{\rightarrow} \moment_{0} \frac{\pstar}{1-\pstar}.  
\end{align*}
Similarly,
\begin{align*}
\frac{1}{N} \sum_{i=1}^N \frac{1-p_i}{p_i} \Yione^2 \stackrel{\mathbb{P}}{\rightarrow} \moment_{1} \frac{1-\pstar}{\pstar} \qquad \text{and} \qquad \frac{2}{N} \sum_{i=1}^N \Yione\Yizero \xrightarrow{} 2\crossmoment.
\end{align*}
Overall we have,
\begin{align*}
\sum_{i=1}^N \E\left[\xi_i^2 \,\middle|\, \sigmafield \right] \stackrel{\mathbb{P}}{\rightarrow} {\moment_{0} \frac{\pstar}{1-\pstar} + \moment_{1} \frac{1 - \pstar}{\pstar} + 2\crossmoment}.   
\end{align*}

\vspace{10pt}
\noindent 
\underline{\textbf{For Weakly Stable Design:}} Under weak design stability, it follows that 
\begin{align*}
\frac{1}{N}\sum_{i=1}^N\frac{1-p_i}{p_i} \stackrel{\mathbb{P}}{\rightarrow} \frac{1-\ponestar}{\ponestar} \qquad \text{and} \qquad \frac{1}{N}\sum_{i=1}^N\frac{p_i}{1-p_i} \stackrel{\mathbb{P}}{\rightarrow} \frac{\ptwostar}{1-\ptwostar}.
\end{align*}
Therefore, Assumption~\ref{assn:IPW} and Lemma~\ref{lemma:aiibii} give
\begin{align*}
\sum_{i=1}^N \E\left[\xi_i^2 \,\middle|\, \sigmafield \right] \stackrel{\mathbb{P}}{\rightarrow} {\moment_{0} \frac{\ptwostar}{1-\ptwostar} + \moment_{1} \frac{1 - \ponestar}{\ponestar} + 2\crossmoment}.   
\end{align*}
Combining the two cases, the asymptotic variance of the IPW estimator is
\begin{align}
\vipw =
\begin{cases}
\vipwstrong = \moment_{0}\dfrac{\pstar}{1-\pstar} + \moment_{1}\dfrac{1-\pstar}{\pstar} + 2\crossmoment 
& \text{(strong design stability)}, \\[1.2em]
\vipwweak = \moment_{0}\dfrac{\ptwostar}{1-\ptwostar} + \moment_{1}\dfrac{1-\ponestar}{\ponestar} + 2\crossmoment 
& \text{(weak design stability)}.
\end{cases}
\label{eqn:vipw_cases_proof}
\end{align}

\medskip 
Next, the boundedness of $p_i$ and $Y_i$ ensures that
\begin{align*}
    |\xi_i| = \left|\frac{K_i-p_i}{\sqrt{N}}\left( \frac{\Yizero}{1-p_i} + \frac{\Yione}{p_i}\right)\right|
    = \frac{1}{\sqrt{N}}|K_i-p_i|\left|\frac{\Yizero}{1-p_i} + \frac{\Yione}{p_i}\right| \leq \frac{4M}{\sqrt{N}\delta}.
\end{align*}
Fix $\varepsilon > 0.$ For any $N > \left(\tfrac{4M}{\delta\varepsilon}\right)^2,$ we have $\mathbf{1}_{\{|\xi_i| > \varepsilon\}} = 0$ a.s. Consequently, for such $N,$
\begin{align*}
    \sum_{i=1}^N \E{\left[\xi_i^2\mathbf{1}_{\{|\xi_i| > \varepsilon\}} \,\middle|\, \sigmafield\right]} = 0,
\end{align*}
and therefore
\begin{align*}
    \lim_{N \to \infty} \sum_{i=1}^N \E{\left[\xi_i^2\mathbf{1}_{\{|\xi_i| > \varepsilon\}} \,\middle|\, \sigmafield\right]} = 0,
\end{align*}
which verifies the Lindeberg condition. Putting together the pieces and  invoking martingale central limit theorem~\cite[Chapter 3]{Hall1980-ny} yields $\sqrt{N}\left(\snipw - \ate\right) \xrightarrow{d} \mathcal{N}\left(0, \vipw\right),$
with asymptotic variance $\vipw$ as specified in~\eqref{eqn:vipw_cases_proof}.

\subsection{Proof of Theorem~\ref{corr:vipw_est} \\ (Variance estimation of the IPW estimator under strong design stability)}

\label{sec:Proof-of-Corr:vipw_est}


\vspace{10pt}

\noindent
First, we establish the the consistency of $\widehat{m}_0^2$ and $\widehat{m}_1^2$. Before doing this, we first show that
\begin{align}
\frac{N_1}{N} 
= \frac{1}{N}\sum_{i=1}^N K_i 
\stackrel{\mathbb{P}}{\rightarrow} \pstar.
\label{eqn:ki}
\end{align}
We decompose
\begin{align*}
\frac{1}{N}\sum_{i=1}^N K_i
= \frac{1}{N}\sum_{i=1}^N (K_i - p_i) 
  + \frac{1}{N}\sum_{i=1}^N p_i.
\end{align*}
Under a strongly stable design, since $p_i \stackrel{\mathbb{P}}{\rightarrow} \pstar$, the second term, being the Cesàro mean of the sequence $\{p_i\}_{i \geq 1}$, also converges in probability to $\pstar$.
Hence, it remains to show that
\begin{align}
\label{eqn:pstarconsistencyeqn2}
\frac{1}{N}\sum_{i=1}^N (K_i - p_i) \stackrel{\mathbb{P}}{\rightarrow} 0.
\end{align}
Since $\E\left[K_i - p_i \,\middle|\, \sigmafield\right] = 0$, the summands form a martingale difference sequence. Moreover, by uniform boundedness of $p_i$ and Chebyshev’s inequality, it follows that for any fixed~$\varepsilon > 0$,
\begin{align*}
\Prob\left(\left|\frac{1}{N}\sum_{i=1}^N(K_i-p_i)\right| \geq \varepsilon\right) \leq  \frac{1}{N^2\varepsilon^2}\sum_{i=1}^N\E\!\left[(K_i - p_i)^2\right]
= \frac{1}{N^2\varepsilon^2} \sum_{i=1}^N \E\!\left[p_i(1 - p_i)\right]
\leq \frac{(1-\delta)^2}{N\varepsilon^2} \to 0,
\end{align*} 
implying claim~\eqref{eqn:pstarconsistencyeqn2}. 

\vspace{10pt}
\noindent\textbf{\underline{Consistency of $\widehat{m}^2_{1}$ \& $\widehat{m}^2_{0}$:}}   We now establish the consistency of $\widehat{m}^2_{1}$; the argument for $\widehat{m}^2_{0}$ follows analogously. Recalling from the main paper,
\begin{align*}
\widehat{m}^2_{1}
= \frac{1}{\max\{N_1,1\}}\sum_{i=1}^N K_i \Yione^2,
\quad \text{where} \quad N_1 = \sum_{i=1}^N K_i.
\end{align*}
To establish the consistency of $\widehat{m}_1^2$, it suffices to show that
\begin{align}
\frac{1}{N}\sum_{i=1}^N K_i \Yione^2 \stackrel{\mathbb{P}}{\rightarrow} \pstar \moment_{1}.
\label{eqn:show}
\end{align}
Note that
\begin{align*}
\frac{\max\{N_1,1\}}{N}
= \frac{N_1}{N} + \frac{1}{N}\,\mathbf{1}_{\{N_1=0\}},
\end{align*}
where the second term converges to zero as $N \to \infty$, and $\tfrac{N_1}{N} \stackrel{\mathbb{P}}{\rightarrow} \pstar$ by~\eqref{eqn:ki}. Hence,  
\begin{align}
\label{eqn:maxn1}
\frac{\max\{N_1,1\}}{N} \stackrel{\mathbb{P}}{\rightarrow} \pstar.
\end{align}
Therefore, once~\eqref{eqn:show} holds, Slutsky’s theorem implies the consistency of $\widehat{m}_1^2$. 

Note that $\tfrac{1}{N}\sum_{i=1}^N (K_i - p_i)\Yione^2$ is sum of a martingale difference sequence.
By the boundedness of \(p_i\) and \(Y_i\), together with Chebyshev’s inequality, for any fixed $\varepsilon > 0$,
\begin{align*}
\Prob\left(\left|\frac{1}{N}\sum_{i=1}^N (K_i - p_i)\Yione^2\right| \geq \varepsilon\right) \leq \frac{1}{N^2\varepsilon^2} \sum_{i=1}^N  \E\left[(K_i - p_i)^2 \Yione^4\right]
\leq \frac{M^4(1-\delta)^2}{N\varepsilon^2} \to 0,
\end{align*}
thus implying 
\begin{align}
\label{eqn:m1consistency1}
\frac{1}{N}\sum_{i=1}^N (K_i - p_i)\Yione^2 \stackrel{\mathbb{P}}{\rightarrow} 0.
\end{align}
By strong stability of the design and Assumption~\ref{assn:IPW}, in conjunction with Lemma~\ref{lemma:aibi}, we have 
\begin{align}
\label{eqn:m1consistency2}
\frac{1}{N}\sum_{i=1}^N p_i \Yione^2 \stackrel{\mathbb{P}}{\rightarrow} \pstar \moment_{1}.
\end{align}
Combining implications~\eqref{eqn:m1consistency1} and~\eqref{eqn:m1consistency2} then gives
\begin{align*}
\frac{1}{N}\sum_{i=1}^N K_i \Yione^2 \stackrel{\mathbb{P}}{\rightarrow} \pstar \moment_{1},
\end{align*}
proving our claim~\eqref{eqn:show}, and thereby establishing the consistency of $\widehat{m}^2_{1}$.

\vspace{10pt}
\noindent\textbf{\underline{Bounding the cross-moment term:}}  By the Cauchy-Schwarz inequality,
\begin{align*}
\left|\frac{1}{N}\sum_{i=1}^N \Yizero\Yione \right|
\leq \left( \frac{1}{N}\sum_{i=1}^N \Yizero^2 \right)^{1/2}
\left( \frac{1}{N}\sum_{i=1}^N \Yione^2 \right)^{1/2}.
\end{align*}
Taking the limit as $N \to \infty$ yields
\begin{align}
|\crossmoment| \leq m_{0} m_{1}.
\label{eqn:inequality2}
\end{align}
Hence, under strong design stability, we have
\begin{align}
\vipwstrong
= \moment_{0}\frac{\pstar}{1-\pstar}
+ \moment_{1}\frac{1-\pstar}{\pstar}
+ 2\crossmoment 
\leq \left(m_{0}\sqrt{\frac{\pstar}{1-\pstar}}
+ m_{1}\sqrt{\frac{1-\pstar}{\pstar}}\right)^2.
\label{eqn:inequality1}
\end{align}
Since $\mhat_0^2$ and $\mhat_1^2$ are consistent for $\moment_0$ and $\moment_1$; non-negativity and the continuous mapping theorem ensures $\mhat_j \stackrel{\mathbb{P}}{\rightarrow} m_j$ for $j\in \{0,1\}$, and hence the variance estimator 
\begin{align*}
\widehat{\vipwstronghat} 
= \left(\widehat{m}_0 \sqrt{\frac{\pstar}{1-\pstar}} 
+ \widehat{m}_1 \sqrt{\frac{1-\pstar}{\pstar}}\right)^2
\end{align*}
is a consistent estimator of 
\begin{align*}
\left(m_{0}\sqrt{\tfrac{\pstar}{1-\pstar}} 
+ m_{1}\sqrt{\tfrac{1-\pstar}{\pstar}}\right)^2,
\end{align*} and hence implies $\widehat{\vipwstrong}$ conservatively estimates $\vipwstrong.$

For $\widehat{\vipwstrong}$ to be consistent for $\vipwstrong$, equality must hold in inequality~\eqref{eqn:inequality1}. This corresponds to the equality case of the Cauchy–Schwarz inequality in~\eqref{eqn:inequality2}, which implies
\begin{align*}
\frac{\Yione}{\Yizero} = c \quad \text{for all } i,
\end{align*}
for some constant $c \in \mathbb{R}$, and hence the potential outcomes are additive on the log scale, i.e., satisfy~\eqref{eqn:log-additive-model}. This completes the proof of the theorem.

\vspace{10pt}
\noindent\textbf{\underline{Consistency of $\widehat{p}^\star$ (proposed in Remark~\ref{remark:phat} under unknown $\pstar$:}}
Here we show that when $\pstar$ is unknown or difficult to compute explicitly, the estimator $\widehat{p}^\star$ proposed in Remark~\ref{remark:phat} is consistent for $\pstar$.
Recall that under the sequential treatment assignment,
\begin{align*}
p_i \in \sigmafield 
\quad \text{and} \quad 
\Prob(K_i = 1 \mid \sigmafield) = p_i,
\end{align*}
where $\sigmafield = \sigma(K_1, Y_1, \ldots, K_{i-1}, Y_{i-1})$ denotes the sigma-field generated by the past treatment assignments and outcome history. Hence, given the past history, the inclusion probabilities $p_i$ are known to the experimenter. Under strong design stability, we have $p_i \stackrel{\mathbb{P}}{\rightarrow} \pstar$. Consequently, by Lemma~\ref{lemma:cesaro_inprob} 
\begin{align}
\widehat{p}^\star 
= \frac{1}{N}\sum_{i=1}^N p_i 
\stackrel{\mathbb{P}}{\rightarrow} \pstar,
\label{eqn:pwidehatstarconsistency}
\end{align}
establishing the consistency of $\widehat{p}^\star$. Hence, $\widehat{p}^\star$ can be substituted into $\widehat{\vipwstrong}$, which would still consistently estimate $\left(m_{0}\sqrt{\tfrac{\pstar}{1-\pstar}} + m_{1}\sqrt{\tfrac{1-\pstar}{\pstar}}\right)^2$. The remaining arguments then follow analogously to the case with known $\pstar$.




\subsection{Proof of Theorem~\ref{corr:vipw_est_weak} \\ (Variance estimation of the IPW estimator under weak design stability)}
\label{sec:Proof-of-Corr:vipw_est_weak}
We first consider the case where~$\ponestar, \ptwostar$ and $\widetilde{p}$ are known.

\vspace{10pt} 
\noindent\textbf{\underline{Consistency of $\widetilde{m}_0^2$ and $\widetilde{m}_1^2$:}} We establish the consistency of $\widetilde{m}_1^2$; the proof for $\widetilde{m}_0^2$ follows analogously. 
Under the additional restriction~\eqref{eqn:extra_assumption} of the main paper,
\begin{align*}
\frac{1}{N}\sum_{i=1}^N p_i \stackrel{\mathbb{P}}{\rightarrow} \widetilde{p},
\end{align*}
Hence, under weak design stability, Assumption~\ref{assn:IPW} together with Lemma~\ref{lemma:aiibii} implies that
\begin{align}
\frac{1}{N}\sum_{i=1}^N K_i \Yione^2 \stackrel{\mathbb{P}}{\rightarrow} \widetilde{p}\,\moment_{1}.
\label{eqn:proof_m1_part1}
\end{align}
Since $\widetilde{p}$ is known, $\widetilde{m}_1^2 = \tfrac{1}{N\widetilde{p}} \sum_{i=1}^N K_i\Yione^2$ serves as a consistent estimator of $\moment_1$.

\vspace{10pt} 
\noindent\textbf{\underline{Bounding the cross-moment term:}} We have already established in~\eqref{eqn:inequality2} that $|\crossmoment| \leq m_{0} m_{1}.$
Hence, under weak design stability,
\begin{align}
\vipwweak
= \moment_{0}\frac{\ptwostar}{1-\ptwostar} 
  + \moment_{1}\frac{1-\ponestar}{\ponestar} 
  + 2\crossmoment
\leq \moment_{0}\frac{\ptwostar}{1-\ptwostar} 
  + \moment_{1}\frac{1-\ponestar}{\ponestar} 
  + 2 m_{0} m_{1}.
\label{eqn:inequality3}
\end{align}
By arguments analogous to those in the proof of Theorem~\ref{corr:vipw_est}, 
$\widehat{\vipwweakhat}$ consistently estimates 

$\left(\moment_{0}\tfrac{\ptwostar}{1-\ptwostar} 
+ \moment_{1}\tfrac{1-\ponestar}{\ponestar} 
+ 2 m_{0}m_{1}\right)$ 
under weak design stability. 
Analogous arguments as in the proof of Theorem~\ref{corr:vipw_est} 
together with inequality~\eqref{eqn:inequality3}, 
yields the desired result, with consistency attained when the potential outcomes are additive on log scale, i.e., satisfy~\eqref{eqn:log-additive-model}. This completes the proof of the theorem.

\vspace{10pt}


\noindent\textbf{\underline{Consistency of $\widehat{p}_1^\star$, $\widehat{p}_2^\star$, and $\widetildepest$ 
(proposed in Remark~\ref{remark:ipw_weak_remark}) under unknown $\ponestar$, $\ptwostar$, and $\widetilde{p}$:
}} 

\vspace{3pt}
\noindent Under weak design stability,
\begin{align*}
\frac{1}{N} \sum_{i = 1}^N \frac{1}{p_i} \stackrel{\mathbb{P}}{\rightarrow} \frac{1}{\ponestar}
\qquad \text{and} \qquad
\frac{1}{N} \sum_{i = 1}^N \frac{1}{1 - p_i} \stackrel{\mathbb{P}}{\rightarrow} \frac{1}{1 - \ptwostar}.
\end{align*}
Moreover, since $p_i \in \sigmafield$, the current inclusion probability is known to the experimenter given the past assignment history and potential outcomes. Hence, $\tfrac{1}{N}\sum_{i=1}^N\tfrac{1}{p_i}$ and $\tfrac{1}{N}\sum_{i=1}^N\tfrac{1}{1-p_i}$ can be viewed as consistent estimators of $\tfrac{1}{\ponestar}$ and $\tfrac{1}{1-\ptwostar}$, respectively. Finally, Assumption~\ref{assn:IPW}(a) together with the continuous mapping theorem implies that $\widehat{p}_1^\star$ and $\widehat{p}_2^\star$ as in Remark~\ref{remark:ipw_weak_remark} consistently estimate $\ponestar$ and $\ptwostar$, respectively. By similar reasoning, $\widetilde{p}$ can be consistently estimated by $\widetildepest = \tfrac{1}{N}\sum_{i=1}^N p_i$, under the additional restriction~\eqref{eqn:extra_assumption}.


\subsection{Proof of Theorem~\ref{Thm:aipw} (CLT for the AIPW estimator)}
\label{sec:Proof-of-Thm:aipw}


The proof of this theorem differs from that of Theorem~\ref{Thm:ipw}, as it is not straightforward to apply the martingale central limit theorem~\cite[Chapter 3]{Hall1980-ny} directly. Instead, we first analyze a proxy estimator $\tnaipw$ defined  as
\begin{align*}
\tnaipw = \frac{1}{N} \sum_{i=1}^{N} \left[\left\{  
    \frac{K_i\left(\Yione - \overline{Y}_{i-1}(1)\right) }{p_i} + \overline{Y}_{i-1}(1) 
\right\}
- \left\{
    \frac{(1 - K_i)\left(\Yizero - \overline{Y}_{i-1}(0)\right) }{1 - p_i} + \overline{Y}_{i-1}(0) 
\right\} \right],
\end{align*}
where $\overline{Y}_{i-1}(l) = \tfrac{1}{i-1} \sum_{j<i} Y_j(l), \text{ for } l \in \{0, 1\}.$
The analytically tractable estimator $\tnaipw$, though not directly estimable from the observed data, is constructed to closely mimic the behavior of the actual estimator $\snaipw$. A central limit theorem for $\tnaipw$ can be established using the martingale central limit theorem~\cite[Chapter 3]{Hall1980-ny}. The crucial step then is to show that the difference between $\snaipw$ and $\tnaipw$ is asymptotically negligible, in the sense that
\begin{align}
\frac{\E\left[\snaipw-\tnaipw\right]^2}{\Var\left[\tnaipw\right]} \to 0 \text{ as } N \to \infty.
\label{eqn:hajek}   
\end{align}
This allows us to invoke H\'{a}jek's Lemma (see Lemma~\ref{lemma:hajek_proof}), which implies that the asymptotic distribution of $\snaipw$ matches that of $\tnaipw$, thereby establishing the central limit theorem for $\snaipw$. 

We start by observing that $\tnaipw$ after proper centering and scaling can be written as a sum of martingale difference sequence. 
\begin{align*}
\sqrt{N}\left(\tnaipw - \ate\right)
&= \sum_{i=1}^N \frac{K_i - p_i}{\sqrt{N}}\left( \frac{\Yizero - \overline{Y}_{i-1}(0)}{1 - p_i} + \frac{\Yione - \overline{Y}_{i-1}(1)}{p_i}  \right) \;=\; \sum_{i=1}^N \ze_i,
\end{align*}
where $\ze_i = \tfrac{K_i - p_i}{\sqrt{N}}\left( \tfrac{\Yizero - \overline{Y}_{i-1}(0)}{1 - p_i} + \tfrac{\Yione - \overline{Y}_{i-1}(1)}{p_i}  \right).$
Now,
\begin{align*}
\E\left[\ze_i \,\middle|\, \sigmafield\right] 
&= \frac{1}{\sqrt{N}}\left(\frac{\Yizero - \overline{Y}_{i-1}(0)}{1 - p_i} + \frac{\Yione - \overline{Y}_{i-1}(1)}{p_i}\right)
\E\left[ K_i - p_i \,\middle|\, \sigmafield \right]
= 0,
\end{align*}
implying $\{\ze_i\}_{i=1}^N$ are terms of a martingale difference sequence and that $\tnaipw$ is an unbiased estimator for $\ate$. The  total conditional variance of $\{ \ze_i  \}_{i \geq 1}$ is given by 
\begin{align*}
\sum_{i=1}^N \E\left[\ze_i^2 \,\middle|\, \sigmafield\right] 
&= \sum_{i=1}^N \frac{1}{N} \left(\frac{\Yizero - \overline{Y}_{i-1}(0)}{1 - p_i} + \frac{\Yione - \overline{Y}_{i-1}(1)}{p_i}\right)^2\E\left[(K_i-p_i)^2\,\middle|\,\sigmafield\right] \\ 
&= \frac{1}{N}\sum_{i=1}^N \frac{p_i}{1-p_i}\left(\Yizero - \overline{Y}_{i-1}(0)\right)^2 + \frac{1}{N}\sum_{i=1}^N \frac{1-p_i}{p_i}\left(\Yione - \overline{Y}_{i-1}(1)\right)^2 \\  &+ \frac{2}{N}\sum_{i=1}^N\left(\Yizero - \overline{Y}_{i-1}(0)\right)\left(\Yione - \overline{Y}_{i-1}(1)\right).
\end{align*}
Next, we verify that the total conditional variance converges in probability to a constant and that the Lindeberg condition holds.

\vspace{10pt}
\noindent 
\underline{\textbf{For Strongly Stable Design:}} Under strong design stability and Assumption~\ref{assn:AIPW}(a), the continuous mapping theorem implies that $\tfrac{p_i}{1-p_i} \stackrel{\mathbb{P}}{\rightarrow} \tfrac{\pstar}{1-\pstar}.$
Therefore, Lemma~\ref{lemma:derived_from_ass_aipw_1}, together with Lemma~\ref{lemma:aibi} and Assumption~\ref{assn:AIPW}(a)–(b), implies
\begin{align*}
\frac{1}{N}\sum_{i=1}^N \frac{p_i}{1-p_i}\left(\Yizero - \overline{Y}_{i-1}(0)\right)^2 \stackrel{\mathbb{P}}{\rightarrow} \Yvar_{0}\frac{\pstar}{1-\pstar}.
\end{align*}Similar arguments yield, 
\begin{align*}
\frac{1}{N}\sum_{i=1}^N \frac{1-p_i}{p_i}\left(\Yione - \overline{Y}_{i-1}(1)\right)^2 &\stackrel{\mathbb{P}}{\rightarrow} \Yvar_{1}\frac{1-\pstar}{\pstar}, \\ \text{and} \qquad
 \frac{2}{N}\sum_{i=1}^N \left(\Yizero - \overline{Y}_{i-1}(0)\right)\left(\Yione - \overline{Y}_{i-1}(1)\right) &\xrightarrow{} 2\Ycov.
\end{align*} Overall, 
\begin{align*}
\sum_{i=1}^N \E\left[\ze_i^2 \,\middle|\, \sigmafield\right] \stackrel{\mathbb{P}}{\rightarrow} \Yvar_{0}\frac{\pstar}{1-\pstar} + \Yvar_{1}\frac{1-\pstar}{\pstar} + 2\Ycov.
\end{align*}

\vspace{10pt}
\noindent 
\underline{\textbf{For Weakly Stable Design:}} By arguments analogous to the weakly stable case in the proof of Theorem~\ref{Thm:ipw}, it follows that 
\begin{align*}
\sum_{i=1}^N \E\left[\ze_i^2 \,\middle|\, \sigmafield\right] \stackrel{\mathbb{P}}{\rightarrow} \Yvar_{0}\frac{\ptwostar}{1-\ptwostar} + \Yvar_{1}\frac{1-\ponestar}{\ponestar} + 2\Ycov.
\end{align*} Combining the two cases, the asymptotic variance of the AIPW estimator is
\begin{align}
\vaipw =
\begin{cases}
\vaipwstrong = \Yvar_{0}\dfrac{\pstar}{1-\pstar} + \Yvar_{1}\dfrac{1-\pstar}{\pstar} + 2\Ycov & \text{under strong design stability}, \\[1.2em]
\vaipwweak = \Yvar_{0}\dfrac{\ptwostar}{1-\ptwostar} + \Yvar_{1}\dfrac{1-\ponestar}{\ponestar} + 2\Ycov & \text{under weak design stability}.
\end{cases}
\label{eqn:vaipw_cases_proof}
\end{align}

\vspace{10pt}
Next, note that the uniform boundedness of $p_i$ and $Y_i$ imply that
\begin{align*}
|\ze_i|  
= \frac{1}{\sqrt{N}}|K_i - p_i| \left|
\frac{\Yizero - \overline{Y}_{i-1}(0)}{1 - p_i} + \frac{\Yione - \overline{Y}_{i-1}(1)}{p_i}\right| 
\leq \frac{8M}{\sqrt{N}\delta}.
\end{align*} 
Fix $\varepsilon > 0.$ For any $N > \left(\tfrac{8M}{\delta\varepsilon}\right)^2,$ we have $\mathbf{1}_{\{|\ze_i| > \varepsilon\}} = 0$ a.s. Consequently, for such $N,$
\begin{align*}
    \sum_{i=1}^N \E{\left[\ze_i^2\mathbf{1}_{\{|\ze_i| > \varepsilon\}} \,\middle|\, \sigmafield\right]} = 0,
\end{align*}
and therefore
\begin{align*}
    \lim_{N \to \infty} \sum_{i=1}^N \E{\left[\ze_i^2\mathbf{1}_{\{|\ze_i| > \varepsilon\}} \,\middle|\, \sigmafield\right]} = 0,
\end{align*}
which verifies the Lindeberg condition. Putting together the pieces and applying the martingale central limit theorem \cite[Theorem 3]{Hall1980-ny} yields 
\begin{align}
\sqrt{N}\left(\tnaipw - \ate\right) \xrightarrow{d} \mathcal{N}\left(0, \vaipw\right). \label{clt_tnaipw}
\end{align}
It now remains to verify the condition \eqref{eqn:hajek}. 

\vspace{10pt}
\noindent 
\underline{\textbf{Verifying condition~\eqref{eqn:hajek}:}}
Observe that $\snaipw-\tnaipw = \sum_{i = 1}^N \Delta_i,$
%
%
where
\begin{align*}
    \Delta_i = \frac{K_i-p_i}{N} \left( \frac{\widehat{Y}_{i-1}(1)-\overline{Y}_{i-1}(1)}{p_i} + \frac{\widehat{Y}_{i-1}(0)-\overline{Y}_{i-1}(0)}{1-p_i} \right).
\end{align*} It is easy to verify that $\{ \Delta_i \}_{i \geq 1}$ is a martingale difference with respect to filtration $\{\sigmafield\}_{i \geq 1}$, and hence 
\begin{align}
\E\left[\left(\snaipw-\tnaipw\right)^2\right]
&= \sum_{i=1}^N \E\big[\Delta_i^2\big]. \label{eq:ms_diff}
\end{align}  
Using boundedness of $p_i$,
\begin{align*}
\E\left[\Delta_i^2\right] = \E\left[\E\left[\Delta_i^2 \,\middle|\, \sigmafield\right]\right]
&= \frac{1}{N^2} \E\!\left[ p_i(1-p_i) 
\left( \frac{\widehat{Y}_{i-1}(1)-\overline{Y}_{i-1}(1)}{p_i} 
      + \frac{\widehat{Y}_{i-1}(0)-\overline{Y}_{i-1}(0)}{1-p_i} \right)^{\!2} \right] \\
&\le \frac{(1-\delta)^2}{N^2} \E\!\left[ A_i^2 + 2 A_i B_i + B_i^2 \right],
\end{align*}
where
\begin{align*}
    A_i = \frac{\widehat{Y}_{i-1}(1)-\overline{Y}_{i-1}(1)}{p_i} \qquad \text{and} \qquad  B_i = \frac{\widehat{Y}_{i-1}(0)-\overline{Y}_{i-1}(0)}{1-p_i}.
\end{align*}
Since $\E\left[\frac{Y_j(1)(K_j-p_j)}{p_j}\big|\mathcal{F}_{j-1}\right]=0, A_i$ can be expressed as sum of a martingale difference sequence. Hence, by Assumption~\ref{assn:AIPW}(a)-(b),
\begin{align}
\E\left[A_i^2\right] 
= \frac{1}{(i-1)^2} \sum_{j<i} \E\!\left[ \frac{Y_j(1)^2 (K_j-p_j)^2}{p_i^2p_j^2} \right] 
\le \frac{M^2(1-\delta)^2}{(i-1)\delta^4}.\label{eq:Ai_bound}
\end{align}
Similarly, 
\begin{align}
\E\left[B_i^2\right] \le \frac{M^2(1-\delta)^2}{(i-1)\delta^4}, \quad \text{and} \quad \left|\E\left[A_i B_i\right]\right| \le \sqrt{\E\left[A_i^2\right] \E\left[B_i^2\right]} \le \frac{M^2(1-\delta)^2}{(i-1)\delta^4}.
\label{eq:Bi_bound}
\end{align}
Combining \eqref{eq:ms_diff}-\eqref{eq:Bi_bound}, we have
\begin{align}
\E\left[\left(\snaipw-\tnaipw\right)^2\right]
\le \frac{(1-\delta)^2}{N^2}\sum_{i<N}\frac{4M^2(1-\delta)^2}{(i-1)\delta^4} = \mathcal{O}\!\left( \frac{\log N}{N^2} \right). \label{eq:diff_final_bound}
\end{align}
Assumption~\ref{assn:AIPW}(a)-(b) ensures that 
\begin{align}
\Var\left[\tnaipw\right] 
&= \E\left[ \frac{1}{N^2} \sum_{i=1}^N 
 \left\{ \frac{p_i}{1-p_i} \left(\Yizero-\overline{Y}_{i-1}(0)\right)^2
       + \frac{1-p_i}{p_i} \left(\Yione-\overline{Y}_{i-1}(1)\right)^2 \right. \right. \nonumber \\
&\quad\quad\quad\quad\quad\left.\left. +\, 2 \left(\Yizero-\overline{Y}_{i-1}(0)\right) \left(\Yione-\overline{Y}_{i-1}(1)\right) \right\} \right] \nonumber \\
&\le \frac{1}{N^2} \sum_{i=1}^N \left( \frac{4M^2(1-\delta)}{\delta} + \frac{4M^2\delta}{1-\delta} + 8M^2 \right) = \mathcal{O}\!\left( \frac{1}{N} \right).
\label{eq:var_tnaipw}
\end{align}
Combining the bounds in~\eqref{eq:diff_final_bound} and~\eqref{eq:var_tnaipw}, we conclude that
\begin{align*}
\frac{\E\left[\left(\snaipw-\tnaipw\right)^2\right]}{\Var\left[\tnaipw\right]}
= \mathcal{O}\!\left( \frac{\log N}{N} \right) \to 0 \quad \text{as } N\to\infty,   
\end{align*}
which completes the proof of~\eqref{eqn:hajek}. Hence, it follows that $\sqrt{N}\left(\snaipw - \ate\right) \xrightarrow{d} \mathcal{N}\left(0, \vaipw\right),$
with asymptotic variance $\vaipw$ specified in~\eqref{eqn:vaipw_cases_proof}. 

\subsection{Proof of Theorem~\ref{corr:vaipw_est} \\ (Variance estimation of the AIPW estimator under strong design stability)}
\label{sec:Proof-of-Corr:vaipw_est}

We begin by considering the case in which $\pstar$ is known.

\medskip
\noindent\textbf{\underline{Consistency of $\widehat{\sigma}^2_{1}$ \& $\widehat{\sigma}^2_{0}$:}} Recalling from the main paper, 
\begin{align*}
\widehat{\sigma}^2_1 
&= \frac{1}{\max\{N_1,1\}} \sum_{i=1}^N K_i \left(\Yione - \widehat{Y}_{i-1}(1)\right)^2,
\end{align*}
where $N_1 = \sum_{i=1}^N K_i.$ Set $\widehat{Y}_1(1) = 0$ and for $i \geq 2$,
\begin{align*}
\widehat{Y}_{i-1}(1) 
&= \frac{1}{i-1} \sum_{j=1}^{i-1} \frac{K_j Y_j(1)}{p_j}.
\end{align*}
Under strong design stability, Lemma~\ref{lemma:bigintosmall} and Assumption~\ref{assn:AIPW}(a)–(b), together with Lemma~\ref{lemma:aibi}, imply
\begin{align*}
\frac{1}{N} \sum_{i=1}^N 
  p_i \left(\Yione - \widehat{Y}_{i-1}(1)\right)^2 
\stackrel{\mathbb{P}}{\rightarrow} \pstar \Yvar_{1}.
\end{align*}
Since we have already established in~\eqref{eqn:maxn1} that, under strong design stability, $\tfrac{\max\{N_1,1\}}{N} \stackrel{\mathbb{P}}{\rightarrow} \pstar$, Slutsky’s theorem implies
\begin{align*}
\widehat{\sigma}^2_1 \stackrel{\mathbb{P}}{\rightarrow} \Yvar_{1}.   
\end{align*} The proof for $\widehat{\sigma}^2_0$ is analogous. 

\vspace{10pt}
\noindent\textbf{\underline{Bounding the covariance term:}} By the Cauchy-Schwarz inequality,
\begin{align}
\left| 
  \frac{1}{N} \sum_{i=1}^N 
  \left(\Yizero - \overline{Y}_{i-1}(0)\right)
  \left(\Yione - \overline{Y}_{i-1}(1)\right) 
\right|
&\le 
\left( 
  \frac{1}{N} \sum_{i=1}^N 
  \left(\Yizero - \overline{Y}_{i-1}(0)\right)^2 
\right)^{1/2}
\nonumber\\
&\quad\times
\left( 
  \frac{1}{N} \sum_{i=1}^N 
  \left(\Yione - \overline{Y}_{i-1}(1)\right)^2 
\right)^{1/2}.
\label{boundcs}
\end{align}
Taking limits as $N \to \infty$ yields $|\Ycov| \le \sigma_0 \sigma_1$, and hence
\begin{align}
\vaipwstrong
= \Yvar_{0}\frac{\pstar}{1-\pstar} 
+ \Yvar_{1}\frac{1-\pstar}{\pstar} 
+ 2\Ycov
\le 
\left(\sigma_0\sqrt{\frac{\pstar}{1-\pstar}} 
+ \sigma_1\sqrt{\frac{1-\pstar}{\pstar}}\right)^2.
\label{eqn:inequality4}
\end{align}
Using the consistency results $\widehat{\sigma}_0^2 \stackrel{\mathbb{P}}{\rightarrow} \Yvar_0$, 
and $\widehat{\sigma}_1^2 \stackrel{\mathbb{P}}{\rightarrow} \Yvar_1$, 
together with nonnegativity of $\widehat{\sigma}_0$ and $\widehat{\sigma}_1$ 
and the continuous mapping theorem, 
it follows that $\widehat{\sigma}_j \stackrel{\mathbb{P}}{\rightarrow} \sigma_j$ for $j \in \{0,1\}$. 
Consequently, $\widehat{\vaipwstronghat}$, as defined in Section~\ref{sec:aipw_results} of the main paper, 
consistently estimates 
$\left(\sigma_{0}\sqrt{\frac{\pstar}{1-\pstar}} + \sigma_{1}\sqrt{\frac{1-\pstar}{\pstar}}\right)^{\!2}$. Therefore, $\widehat{\vaipwstrong}$ estimates $\vaipwstrong$ conservatively.

For $\widehat{\vaipwstrong}$ to be consistent for $\vaipwstrong$, equality must hold in inequality~\eqref{eqn:inequality4}. 
This corresponds to the equality case of the Cauchy–Schwarz inequality in~\eqref{boundcs}, which implies that 
\begin{align*}
\frac{\Yione - \overline{Y}_{i-1}(1)}{\Yizero - \overline{Y}_{i-1}(0)} = c \quad \text{for all } i,
\end{align*}
for some constant $c \in \mathbb{R}$. Consequently, the potential outcomes satisfy generalized treatment effect homogeneity (Definition~\ref{definition:additive-model}).  
This completes the proof of the theorem.

\medskip

If $\pstar$ is unknown or difficult to compute, the consistent estimator $\widehat{p}^\star$ defined in Remark~\ref{remark:phat} may be used; see Remark~\ref{rem:APIWunknown} for further details.



\subsection{Proof of Theorem~\ref{corr:vaipw_est_weak} \\ (Variance estimation of the AIPW estimator under weak design stability)}
\label{sec:Proof-of-Corr:vaipw_est_weak}

We first consider the case in which $\ponestar$, $\ptwostar$, and $\widetilde{p}$ are known.

\noindent\textbf{\underline{Consistency of $\widetilde{\sigma}^2_{1}$ \& $\widetilde{\sigma}^2_{0}$:}} The argument follows along the same lines as the proof of Theorem~\ref{corr:vipw_est_weak}. 
We first establish the consistency of $\widetilde{\sigma}_1^2$; the proof for $\widetilde{\sigma}_0^2$ is analogous. 
Assuming the additional condition~\eqref{eqn:extra_assumption}, and invoking Lemma~\ref{lemma:bigintosmall}, weak design stability, and Lemma~\ref{lemma:aiibii}, we obtain
\begin{align}
\frac{1}{N}\sum_{i=1}^N K_i \left(\Yione - \widehat{Y}_{i-1}(1)\right)^2 
\stackrel{\mathbb{P}}{\rightarrow} \widetilde{p}\,\Yvar_{1}.
\label{eqn:proof_sigma1_part1}
\end{align}
Therefore, under known $\widetilde{p},$
\begin{align*}
\widetilde{\sigma}_1^2 
= \frac{1}{N\widetilde{p}}\sum_{i=1}^N K_i \left(\Yione - \widehat{Y}_{i-1}(1)\right)^2 
\stackrel{\mathbb{P}}{\rightarrow} \Yvar_{1},
\end{align*}
establishing the consistency of $\widetilde{\sigma}_1^2.$

\vspace{10pt}
\noindent\textbf{\underline{Bounding the covariance term:}} As established in the proof of Theorem~\ref{corr:vaipw_est}, by the Cauchy–Schwarz inequality, the cross-moment term satisfies 
$|\Ycov| \le \sigma_0\sigma_1$, and hence
\begin{align}
\vaipwweak 
= \Yvar_{0}\frac{\ptwostar}{1-\ptwostar} 
+ \Yvar_{1}\frac{1-\ponestar}{\ponestar} 
+ 2\Ycov
\le 
\Yvar_{0}\frac{\ptwostar}{1-\ptwostar} 
+ \Yvar_{1}\frac{1-\ponestar}{\ponestar} 
+ 2\sigma_0\sigma_1.
\label{eqn:inequality5}
\end{align}
By arguments analogous to those in the proof of Theorem~\ref{corr:vaipw_est}, 
$\widehat{\vaipwweakhat}$ consistently estimates

$\left(\Yvar_{0}\frac{\ptwostar}{1-\ptwostar} 
+ \Yvar_{1}\frac{1-\ponestar}{\ponestar} 
+ 2\sigma_0\sigma_1\right),$ and hence conservatively estimates $\vaipwweak.$$\widehat{\vaipwweak}$ is consistent for $\vaipwweak$ whenever equality holds in~\eqref{eqn:inequality5}, i.e., when the potential outcomes satisfy generalized treatment effect homogeneity (Definition~\ref{definition:additive-model}).

If $\ponestar, \ptwostar$, and $\widetilde{p}$ are unknown or difficult to compute, the consistent estimators $\widehat{p}_1^\star, \widehat{p}_2^\star$, and $\widetildepest$, defined in Remark~\ref{remark:ipw_weak_remark}, may be used; see Remark~\ref{rem:APIWunknown} for further details.


\section{Proofs of Lemmas}
In this section, we present the proofs of our main Lemmas~\ref{lemma:Wei_stability}-\ref{lemma:efron_as_half} along with several additional lemmas.

\subsection{Proof of Lemma~\ref{lemma:Wei_stability}}
\label{sec:Proof-of-Lemma:Wei_stability}
Under Wei’s adaptive coin design~\cite{Wei1978-gd}, the $i$th unit is assigned to treatment with probability  
\begin{align}
\label{eqn:weipi}
   p_i = f\!\left(R_{i-1}\right), 
\end{align}
where $R_{i-1} = \tfrac{D_{i-1}}{i-1}$ denotes the normalized treatment--control imbalance after $(i-1)$ assignments, and $f:[-1,1]\to[0,1]$ is a non-increasing function satisfying $f(0)=\tfrac{1}{2}$ and continuous at zero. To ensure that the variance estimators $\snipw$ and $\snaipw$ are well defined, it is necessary that the assignment probabilities be bounded away from $0$ and $1$.  
If $f$ does not automatically satisfy this condition, we consider its truncated version  
\begin{align}
\label{eqn:truncated-p-wei}
p_i = \min\!\left\{\max\!\left\{f(R_{i-1}),\, \delta\right\},\, 1-\delta\right\}, 
\qquad \delta \in (0, \tfrac{1}{2}],
\end{align}
which guarantees $p_i \in [\delta,\, 1-\delta]$ for all $i$. By Theorem~1 of \cite{Wei1978-gd}, the assignment probabilities in~\eqref{eqn:weipi} satisfy  
\begin{align*}
p_i \stackrel{\mathbb{P}}{\rightarrow} \tfrac{1}{2}.
\end{align*}
Since the truncation in~\eqref{eqn:truncated-p-wei} is a continuous transformation, the continuous mapping theorem implies that the truncated inclusion probabilities also converge in probability to $\tfrac{1}{2}$.  
Therefore, Wei’s adaptive coin design satisfies strong design stability with limiting inclusion probability $\pstar = \tfrac{1}{2}$.


\subsection{Proof of Lemma~\ref{lemma:efron_as_half}}
\label{sec:Proof-of-Lemma:efron_as_half}
We begin by showing that Efron’s biased coin design \cite{Efron1971-yh} satisfies weak stability. 
Suppose a total of $k$ units have been assigned to treatment or control. 
Let $m_k$ and $n_k$ denote, respectively, the number of units assigned to the treatment and control groups, so that $m_k + n_k = k$. 
The corresponding treatment–control imbalance after $k$ assignments is given by $D_k = m_k - n_k$.
Under Efron’s biased coin design ($\eta$) the probability of assigning the $(k+1)$th unit to treatment, denoted by $p_{k+1}$ is given by
\begin{align*}
p_{k+1} =
\begin{cases}
\eta & \text{if } D_k < 0, \\
\frac{1}{2} & \text{if } D_k = 0, \\
1 - \eta & \text{if } D_k > 0.
\end{cases}
\end{align*}
Observe that $\{D_k\}_{k\geq1}$ is a Markov chain and the state space is $\mathbb{Z}.$ Since we can always move from $D_k = a$ to $D_{k+1} = (a-1) \text{ or } D_{k+1} = (a+1)$ in a step, the Markov chain is irreducible.  We begin by recalling Foster’s Theorem \cite{Foster1953-av}, which provides a condition for positive recurrence in Markov chains with a countable state space. 
\begin{theorem}[\cite{Foster1953-av}]
\label{thm:foster}
Consider an irreducible discrete-time Markov chain on a countable state space \( S \), 
with transition probability matrix \( P = (p_{i,j})_{i,j \in S} \), 
where \( p_{i,j} \) denotes the probability of transitioning from state \( i \) to state \( j \). The Markov chain is positive recurrent if and only if there exists a Lyapunov function \( V: S \to \mathbb{R} \), such that \( V(i) \geq 0 \) for all \( i \in S \), and

\begin{align*}
\sum_{j \in S} p_{i,j} V(j) &< \infty \quad \text{for } i \in F, \\
\sum_{j \in S} p_{i,j} V(j) &\leq V(i) - \varepsilon \quad \text{for all } i \notin F,
\end{align*}
for some finite set \( F \subset S \) and strictly positive constant \( \varepsilon > 0 \).
\end{theorem}

We will now show that $\{D_k\}_{k\geq1}$ is positive recurrent using the above theorem. Consider the Lyapunov function $V(s) = |s| \text{ for } s \in S = \mathbb{Z}$, which is non-negative for all $s \in S,$ and take $F = \{0\}.$ For $i \in F \text{ i.e. } i=0,$ 
\begin{align*}
\sum_{j \in S}p_{0,j}V(j) = p_{0,-1}V(-1) + p_{0,1}V(1) = p_{0,-1} + p_{0,1} = 1 < \infty.
\end{align*}
Now we will consider the case $i \notin F \text{ i.e. } i \neq 0.$ If $i>0,$
\begin{align*}
\sum_{j \in S} p_{i,j}V(j) - V(i) = p_{i,{i+1}}V(i+1) + p_{i,{i-1}}V(i-1) - V(i) \\ = (1-\eta)(i+1) + \eta(i-1) - i = 1 - 2\eta < 0. 
\end{align*}
If $i<0,$\begin{align*} 
\sum_{j \in S} p_{i,j}V(j) - V(i) = p_{i,{i+1}}V(i+1) + p_{i,{i-1}}V(i-1) - V(i) \\ = \eta(-i-1) + (1-\eta)(-i+1) - (-i) = 1 - 2\eta < 0.
\end{align*}
Taking $\varepsilon = 2\eta - 1 > 0$ and noting that ${\{D_k\}}_{k \geq 1}$ is an irreducible discrete time Markov chain on the countable state space $\mathbb{Z}$, we conclude that the conditions of Foster's Theorem are satisfied. Therefore, ${\{D_k\}}_{k \geq 1}$ is positive recurrent. Since ${\{D_k\}}_{k \geq 1}$ is irreducible, positive recurrent discrete-time Markov chain, it has unique stationary distribution $\boldsymbol{\pi}$, which we have computed in Section~\ref{sec:stat_distn}. By the mean ergodic theorem,
\begin{align*}
\frac{1}{N} \sum_{i=1}^N \frac{1}{p_i} 
&\xrightarrow{a.s.} \E_{\boldsymbol{\pi}}\!\left[\frac{1}{p_i}\right] \\
&= 2\pi(0) + \sum_{d>0} \frac{\pi(d)}{1 - \eta} + \sum_{d<0} \frac{\pi(d)}{\eta} \\
&= 2\pi(0) + \frac{1 - \pi(0)}{2} \left( \frac{1}{1 - \eta} + \frac{1}{\eta} \right) \\
&= \frac{1 - 4\eta + 12\eta^2 - 8\eta^3}{4\eta^2(1 - \eta)}.
\end{align*}
A similar calculation for $\tfrac{1}{1-p_i}$ yields,
\begin{align*}
\frac{1}{N} \sum_{i=1}^N \frac{1}{1-p_i} 
&\xrightarrow{a.s.} \E_{\boldsymbol{\pi}}\!\left[\frac{1}{1 - p_i}\right] \\
&= 2\pi(0) + \sum_{d>0} \frac{\pi(d)}{\eta} + \sum_{d<0} \frac{\pi(d)}{1 - \eta} \\
&= 2\pi(0) + \frac{1 - \pi(0)}{2} \left( \frac{1}{1 - \eta} + \frac{1}{\eta} \right) \\
&= \frac{1 - 4\eta + 12\eta^2 - 8\eta^3}{4\eta^2(1 - \eta)},
\end{align*}
implying Efron's design is weakly stable with $\ponestar = \tfrac{4\eta^2(1-\eta)}{1-4\eta+12\eta^2-8\eta^3}$ and $\ptwostar = \tfrac{1-4\eta+8\eta^2-4\eta^3}{1-4\eta+12\eta^2-8\eta^3}.$
Moreover, mean ergodic theorem also gives,
\begin{align*}
\frac{1}{N} \sum_{i=1}^N p_i 
&\xrightarrow{a.s.} \E_{\boldsymbol{\pi}}(p_i), \\
\E_{\boldsymbol{\pi}}[p_i] 
&= \frac{\pi(0)}{2} 
  + \sum_{d>0} \pi(d)(1-\eta) 
  + \sum_{d<0} \pi(d)\eta \\
&= \frac{\pi(0)}{2} + \sum_{d>0} \pi(d) \\
&= \frac{\pi(0)}{2} + \frac{1-\pi(0)}{2} \\
&= \frac{1}{2},
\end{align*}
implying Efron's design satisfies the extra condition~\eqref{eqn:extra_assumption}.

\setcounter{lemma}{2}
\begin{lemma}\label{lemma:derived_from_ass_aipw_1}
Under Assumption~\ref{assn:AIPW}(b)-(c), as $N \to \infty$,
\begin{align*}
\frac{1}{N} \sum_{i=1}^N \left(\Yione - \overline{Y}_{i-1}(1)\right)^2 \to \Yvar_1
\quad \text{ and } \quad
\frac{1}{N} \sum_{i=1}^N \left(\Yizero - \overline{Y}_{i-1}(0)\right)^2 \to \Yvar_0.
\end{align*}
\end{lemma}

\begin{proof}
We begin by proving the first part. The proof of the second part proceeds analogously. 
Consider the decomposition
\begin{align*}
\frac{1}{N}\sum_{i=1}^N \left(\Yione - \overline{Y}_{i-1}(1)\right)^2
&= 
\underbrace{\frac{1}{N}\sum_{i=1}^N \left(\Yione - \Ybar_N(1)\right)^2}_{A_N}
+ 
\underbrace{\frac{1}{N}\sum_{i=1}^N \left(\Ybar_N(1) - \overline{Y}_{i-1}(1)\right)^2}_{B_N} \\
&\quad + 
\underbrace{\frac{2}{N}\sum_{i=1}^N 
\left(\Yione - \Ybar_N(1)\right)
\left(\Ybar_N(1) - \overline{Y}_{i-1}(1)\right)}_{C_N}.
\end{align*}
By Assumption~\ref{assn:AIPW}(c), $A_N \to \Yvar_1$.  
Next, we show that $B_N \to 0$. Fix $\varepsilon > 0$.  
By Assumption~\ref{assn:AIPW}(c), $\Ybar_N(1) \to \Ybar_1$, so there exists $K \in \mathbb{N}$ such that for all $i \ge K + 1$,
\begin{align*}
\left|\,\Ybar_N(1) - \overline{Y}_{i-1}(1)\,\right| \le 2\varepsilon.
\end{align*}
Using the boundedness of $\Yione$ (Assumption~\ref{assn:AIPW}(b)), we can decompose $B_N$ as
\begin{align*}
B_N
= \frac{1}{N}\sum_{i=1}^K \left(\Ybar_N(1) - \overline{Y}_{i-1}(1)\right)^2
  + \frac{1}{N}\sum_{i=K+1}^N \left(\Ybar_N(1) - \overline{Y}_{i-1}(1)\right)^2.
\end{align*}
The first term is bounded by $\tfrac{4KM^2}{N}$ and the second by $4\varepsilon^2$, yielding
\begin{align*}
   B_N \le \frac{4KM^2}{N} + 4\varepsilon^2. 
\end{align*}
Letting $N \to \infty$ and subsequently $\varepsilon \downarrow 0$ gives $B_N \to 0$. Finally, by the Cauchy-Schwarz inequality,
\begin{align*}
\left|C_N\right| \le 2 A_N^{1/2} B_N^{1/2} \to 0,
\end{align*}
since $A_N \to \Yvar_1$ and $B_N \to 0$.  
Combining these results yields
\begin{align*}
\frac{1}{N}\sum_{i=1}^N \left(\Yione - \overline{Y}_{i-1}(1)\right)^2 \to \Yvar_1.
\end{align*}
The other part proceeds analogously.
\end{proof}

\begin{lemma}\label{lemma:derived_from_ass_aipw_2}
Under Assumption~\ref{assn:AIPW}(b)-(c), as $N \to \infty$,
\begin{align*}
\frac{2}{N}\sum_{i=1}^N \left(\Yizero-\overline{Y}_{i-1}(0)\right)\left(\Yione-\overline{Y}_{i-1}(1)\right) \to 2\Ycov.
\end{align*}
\end{lemma}

\begin{proof}
Observe, 
\begin{align*}
\frac{2}{N}\sum_{i=1}^N \left(\Yizero - \overline{Y}_{i-1}(0)\right)\left(\Yione - \overline{Y}_{i-1}(1)\right)
&= \frac{2}{N}\sum_{i=1}^N \left(\Yizero - \Ybar_{N}(0)\right)\left(\Yione - \Ybar_{N}(1)\right) \\
&\quad + \frac{2}{N}\sum_{i=1}^N \left(\Ybar_{N}(0) - \overline{Y}_{i-1}(0)\right)\left(\Yione - \Ybar_{N}(1)\right) \\
&\quad + \frac{2}{N}\sum_{i=1}^N \left(\Yizero - \Ybar_{N}(0)\right)\left(\Ybar_{N}(1) - \overline{Y}_{i-1}(1)\right) \\
&\quad + \frac{2}{N}\sum_{i=1}^N \left(\Ybar_{N}(0) - \overline{Y}_{i-1}(0)\right)\left(\Ybar_{N}(1) - \overline{Y}_{i-1}(1)\right).
\end{align*}
As $N \to \infty,$ the first term on the right hand side converges to $2\Ycov$ by Assumption~\ref{assn:AIPW}(c), whereas the remaining terms go to zero under bounds provided by the Cauchy–Schwarz inequality. Hence, 
\begin{align*}
\frac{2}{N}\sum_{i=1}^N \left(\Yizero-\overline{Y}_{i-1}(0)\right)\left(\Yione-\overline{Y}_{i-1}(1)\right) \to 2\Ycov \text{ as } N \to \infty.
\end{align*}
\end{proof}

\begin{lemma}\label{lemma:bigintosmall}
Under Assumption~\ref{assn:AIPW}(b)-(c), as $N \to \infty$,
\begin{align*}
\frac{1}{N} \sum_{i=1}^N \left(\Yione - \widehat{Y}_{i-1}(1)\right)^2 \to \Yvar_1
\quad \text{ and } \quad
\frac{1}{N} \sum_{i=1}^N \left(\Yizero - \widehat{Y}_{i-1}(0)\right)^2 \to \Yvar_0.
\end{align*}
\end{lemma}

\begin{proof}
We first establish the result for the first part; the proof of the second part follows analogously. We first show that
\begin{align}
\label{eqn:hatbar}
\widehat{Y}_{i-1}(1) - \overline{Y}_{i-1}(1) \stackrel{\mathbb{P}}{\rightarrow} 0.
\end{align}
Note that,
\begin{align*}
\widehat{Y}_{i-1}(1) - \overline{Y}_{i-1}(1)
= \frac{1}{i-1} \sum_{j=1}^{i-1} \frac{(K_j - p_j) Y_j(1)}{p_j}.
\end{align*}
Since $\E\!\left[\frac{(K_j - p_j) Y_j(1)}{p_j} \,\middle|\, \mathcal{F}_{j-1}\right] = 0$, the summands form a martingale difference sequence. 
Under Assumption~\ref{assn:AIPW}(a)–(b),
\begin{align*}
\sum_{j=1}^{i-1} \frac{1}{(i-1)^2} 
\E\!\left[\!\left(\frac{(K_j - p_j) Y_j(1)}{p_j}\right)^{\!2}\right]
&= \sum_{j=1}^{i-1} \frac{1}{(i-1)^2} 
\E\!\left[\frac{Y_j(1)^2 p_j(1 - p_j)}{p_j^2}\right] \\
&\leq \frac{M^2 \delta}{(i-1)(1 - \delta)}.
\end{align*}
Hence, by Chebyshev’s inequality, claim~\eqref{eqn:hatbar} follows.

Next, consider
\begin{align*}
&\frac{1}{N} \sum_{i=1}^N \!\left[\!
  \left(\Yione - \overline{Y}_{i-1}(1)\right)^2 
  - \left(\Yione - \widehat{Y}_{i-1}(1)\right)^2
\!\right] \\
&= \frac{1}{N} \sum_{i=1}^N 
  \left(\overline{Y}_{i-1}(1) - \widehat{Y}_{i-1}(1)\right)
  \left(\overline{Y}_{i-1}(1) + \widehat{Y}_{i-1}(1) - 2\Yione\right).
\end{align*}
By boundedness of $\Yione$ and $p_i$,
\begin{align*}
\bigl|\overline{Y}_{i-1}(1) + \widehat{Y}_{i-1}(1) - 2\Yione\bigr| 
\leq 3M + \frac{M}{\delta}.    
\end{align*}
Hence, for any $\varepsilon > 0$,
\begin{align}
\label{eqn:bound}
\Prob\!\left[\!
\left|\frac{1}{N} \sum_{i=1}^N 
  \left(\Yione - \overline{Y}_{i-1}(1)\right)^2 
  - \frac{1}{N} \sum_{i=1}^N 
  \left(\Yione - \widehat{Y}_{i-1}(1)\right)^2
\right| \geq \varepsilon
\!\right]
\leq 
\Prob\!\left[\!
\frac{C}{N} \sum_{i=1}^N 
  \left|\overline{Y}_{i-1}(1) - \widehat{Y}_{i-1}(1)\right| 
\geq \varepsilon
\!\right],
\end{align}
where $C = 3M + \tfrac{M}{\delta}$. 
By Lemma~\ref{lemma:cesaro_inprob} and~\eqref{eqn:hatbar}, 
the upper bound in~\eqref{eqn:bound} converges to zero as $N \to \infty$, giving
\begin{align}
\label{eqn:bound1}
\frac{1}{N} \sum_{i=1}^N 
  \left(\Yione - \overline{Y}_{i-1}(1)\right)^2 
- \frac{1}{N} \sum_{i=1}^N 
  \left(\Yione - \widehat{Y}_{i-1}(1)\right)^2 
\stackrel{\mathbb{P}}{\rightarrow} 0.
\end{align}
Lemma~\ref{lemma:derived_from_ass_aipw_1} and~\eqref{eqn:bound1} then imply
\begin{align}
\frac{1}{N} \sum_{i=1}^N 
  \left(\Yione - \widehat{Y}_{i-1}(1)\right)^2 
\stackrel{\mathbb{P}}{\rightarrow} \Yvar_{1}.
\label{eqn:fact1}
\end{align}    
\end{proof}

\begin{lemma}[H\'{a}jek's Lemma]
\label{lemma:hajek_proof} 
Let $\{S_n\}_{n \geq 1}$ and $\{T_n\}_{n \geq 1}$ be sequences of random variables, and let $L$ be a random variable. 
If 
\begin{align}
\frac{T_n - \E(T_n)}{\sqrt{\Var(T_n)}} \;\xrightarrow{d}\; L
\quad\text{and}\quad
\frac{\E\!\left[(T_n - S_n)^2\right]}{\Var(T_n)} \;\to\; 0, \label{eqn:condition}
\end{align}
then
\begin{align*}
\frac{S_n - \E(S_n)}{\sqrt{\Var(S_n)}} \;\xrightarrow{d}\; L.
\end{align*}
\end{lemma}

\begin{proof}
We first compare the standardized versions of $T_n$ and $S_n$ under the variance of $T_n$. Observe that
\begin{align*}
\E\!\left[
  \left(
    \frac{T_n - \E(T_n)}{\sqrt{\Var(T_n)}}
    - \frac{S_n - \E(S_n)}{\sqrt{\Var(T_n)}}
  \right)^{\!2}
\right]
&= \frac{\E\!\left[\,(T_n - S_n) - \E(T_n - S_n)\,\right]^{\!2}}{\Var(T_n)} 
  \;\le\; \frac{\E\!\left[(T_n - S_n)^2\right]}{\Var(T_n)} 
  \;\to\; 0.
\end{align*}
Hence,
\begin{align}
\frac{T_n - \E(T_n)}{\sqrt{\Var(T_n)}}
- \frac{S_n - \E(S_n)}{\sqrt{\Var(T_n)}} \stackrel{\mathbb{P}}{\rightarrow} 0,
\quad\text{and thus}\quad
\frac{S_n - \E(S_n)}{\sqrt{\Var(T_n)}} \xrightarrow{d} L.
\label{eqn:ratio0_refined}
\end{align}To replace $\Var(T_n)$ by $\Var(S_n)$ in the denominator, note that
\begin{align*}
\frac{\E[(T_n - S_n)^2]}{\Var(T_n)}
&\ge \frac{\Var(T_n - S_n)}{\Var(T_n)}
 = \frac{\Var(T_n) + \Var(S_n) - 2\Cov(T_n, S_n)}{\Var(T_n)}.
\end{align*}
By the Cauchy-Schwarz inequality, $\Cov(T_n, S_n) \le \sqrt{\Var(T_n)\Var(S_n)}$, so
\begin{align*}
\frac{\E[(T_n - S_n)^2]}{\Var(T_n)}
&\ge \left(1 - \sqrt{\frac{\Var(S_n)}{\Var(T_n)}}\right)^{\!2}.
\end{align*}
Since the left-hand side tends to zero by the condition~\eqref{eqn:condition}, it follows that
\begin{align}
\frac{\Var(S_n)}{\Var(T_n)} \to 1.
\label{eqn:ratio1_refined}
\end{align}
Combining~\eqref{eqn:ratio0_refined} and~\eqref{eqn:ratio1_refined} with Slutsky’s theorem yields
\[
\frac{S_n - \E(S_n)}{\sqrt{\Var(S_n)}} \xrightarrow{d} L,
\]
as required.
\end{proof}

\section{On the lack of ordering in the efficiency of variance estimators}
\label{sec:onthelack}
As noted in Remark~\ref{rem:APIWunknown}, $\vaipw \leq \vipw$, implying that $\snaipw$ is more efficient than $\snipw$. A natural question is whether the corresponding variance estimators exhibit a similar ordering. In general, this need not be the case: the variance estimators may not be ordered in terms of conservativeness, and their behaviour depends on the potential outcomes. Without additional assumptions on the performance of the regression adjustment, no general ordering can be established.

Consider the IPW estimator \(\snipw\) and the AIPW estimator \(\snaipw\) under a strongly stable design. Recall, as defined in our paper,
\begin{align*}
\widehat{\vipwstronghat}
= \left(\widehat{m}_0\sqrt{\frac{\pstar}{1-\pstar}}+\widehat{m}_1\sqrt{\frac{1-\pstar}{\pstar}}\right)^2,
\end{align*}
where \(\widehat{m}_0^2=\frac{1}{\max\{N_0,1\}}\sum_{i=1}^N(1-K_i)Y_i^2,\) and 
\(\widehat{m}_1^2=\frac{1}{\max\{N_1,1\}}\sum_{i=1}^N K_iY_i^2\), with
\(N_1=\sum_{i=1}^N K_i\) and \(N_0=N-N_1\). Similarly,
\begin{align*}
\widehat{\vaipwstronghat}
= \left(\widehat{\sigma}_0\sqrt{\frac{\pstar}{1-\pstar}}+\widehat{\sigma}_1\sqrt{\frac{1-\pstar}{\pstar}}\right)^2,
\end{align*}
where
\begin{align*}
\widehat{\sigma}_0^2=\frac{1}{\max\{N_0,1\}}\sum_{i=1}^N(1-K_i)\bigl(Y_i-\widehat{Y}_{i-1}(0)\bigr)^2,
\quad
\widehat{\sigma}_1^2=\frac{1}{\max\{N_1,1\}}\sum_{i=1}^N K_i\bigl(Y_i-\widehat{Y}_{i-1}(1)\bigr)^2,
\end{align*}
and
\begin{align*}
\widehat{Y}_{i-1}(0)=\frac{1}{i-1}\sum_{j<i}\frac{(1-K_j)Y_j}{1-p_j},
\qquad
\widehat{Y}_{i-1}(1)=\frac{1}{i-1}\sum_{j<i}\frac{K_jY_j}{p_j}.
\end{align*}
Observe that $(1-K_i)Y_i^2 \ge (1-K_i)\bigl(Y_i-\widehat{Y}_{i-1}(0)\bigr)^2$ holds if and only if the regression adjustment $\widehat{Y}_{i-1}(0)$ reduces the variability of the outcome, that is, if the predictor captures variation in $Y_i$ more effectively than the trivial predictor $0$. However, if the regression model is misspecified or unstable—for instance, in the early stages of the experiment—the prediction error may exceed the magnitude of the raw outcome. Thus, although one typically expects the AIPW variance estimator to be more efficient than the IPW variance estimator, this is an empirical property of the potential outcomes rather than a guaranteed theoretical result.

\section{Theoretical explanation of observed coverage patterns}
\label{sec:theoretical_explanation_of_observed_coverage_patterns}

The simulation results exhibit an interesting pattern: for very high nominal confidence levels, the variance estimator appears less conservative, with empirical coverage closer to the nominal level. We now show that this behaviour is consistent with the asymptotic theory. 

To illustrate the argument, we focus on the IPW estimator $\snipw$ under a strongly stable design; the other cases follow analogously. Under our assumptions, the asymptotic coverage at nominal level $\alpha\in(0,1)$ is
\begin{eqnarray*}
&& \lim_{N \to \infty}\mathbb{P}\!\left(
\ate \in 
\left[
\snipw 
- z_{1-\alpha/2}
\sqrt{\frac{\widehat{\vipwstrong}}{N}},
\;
\snipw 
+ z_{1-\alpha/2}
\sqrt{\frac{\widehat{\vipwstrong}}{N}}
\right]
\right)
\\[0.3em]
&=&
2\Phi\!\left(
z_{1-\alpha/2}
\sqrt{
\frac{
\left(
m_0\sqrt{\tfrac{\pstar}{1-\pstar}}
+
m_1\sqrt{\tfrac{1-\pstar}{\pstar}}
\right)^2
}{
\vipwstrong
}
}
\right) - 1 .
\end{eqnarray*}
Define $\Delta(\alpha)$ as the difference between asymptotic coverage and the nominal level:
\begin{align*}
   \Delta(\alpha) = 2\Phi\!\left(
z_{1-\alpha/2}
\sqrt{
\frac{
\left(
m_0\sqrt{\frac{\pstar}{1-\pstar}}
+
m_1\sqrt{\frac{1-\pstar}{\pstar}}
\right)^2
}{
\vipwstrong
}
}
\right)-1 - (1-\alpha). 
\end{align*} Differentiating $\Delta(\alpha)$ with respect to $\alpha$ yields

\begin{align*}
    \frac{d}{d\alpha} \Delta(\alpha) = 1 - k \frac{\phi(kz_{1-\alpha/2})}{\phi(z_{1-\alpha/2})},
\end{align*}
where $k = \sqrt{
\frac{
\left(
m_0\sqrt{\frac{\pstar}{1-\pstar}}
+
m_1\sqrt{\frac{1-\pstar}{\pstar}}
\right)^2
}{
\vipwstrong
}
}
.$ This expression simplifies to
\begin{align*}
    \frac{d}{d\alpha} \Delta(\alpha) = 1 - ke^{-\frac{1}{2}(k^2-1)z_{1-\alpha/2}^2}
\end{align*} By the Cauchy--Schwarz inequality, we have $k \ge 1$. The case $k=1$ corresponds to $\Delta(\alpha)\equiv 0$, i.e., the consistent variance estimator, which is exactly in agreement with our theory. When \(k > 1\), note that \(\Delta(\alpha) \to 0\) as \(\alpha \to 0\) and \(\alpha \to 1\). Moreover, since \(z_{1-\alpha/2}\) decreases as \(\alpha\) increases, the exponential term increases in \(\alpha\), implying that \(\frac{d}{d\alpha}\Delta(\alpha)\) changes sign exactly once. Consequently, \(\Delta(\alpha)\) is unimodal: it increases from zero, attains a maximum at some \(\alpha^\star \in (0,1)\), and then decreases back to zero. This explains the pattern observed in the simulation results.

\section{Practical assessment of design stability}
\label{sec:practical_assessment_of_stability}
A natural question, concerns how one might assess in practice whether a given adaptive design satisfies strong or weak stability. While the definitions are asymptotic, several simple diagnostic tools can be implemented when the assignment probabilities ${p_i}$ are explicit functions of past treatment assignments.  

For such designs, one may simulate trajectories of ${p_i}$ under the data-generating mechanism and examine their empirical behaviour as $i$ increases. A first diagnostic is to plot the sequence $p_i$ against the index $i$ across Monte Carlo replicates, which provides a direct visual assessment of whether the assignment probabilities stabilise; such stabilisation is indicative of strong stability.

A second set of diagnostics for weak stability is based on the quantities
\begin{align*} \frac{1}{N} \sum_{i=1}^N \frac{1}{p_i}, \qquad \frac{1}{N} \sum_{i=1}^N \frac{1}{1-p_i}, \end{align*} viewed as functions of $N$. Convergence of these quantities as $N$ increases is indicative of weak stability. 

As an example, we consider the empirical behaviour of the assignment probabilities under Wei’s adaptive coin design~\cite{Wei1978-gd}. For increasing indices $i$, we generate independent Monte Carlo realisations of $p_i$ according to the design. For a fixed tolerance level $\varepsilon > 0$, we estimate
\[
\mathbb{P}\bigl(|p_i - 1/2| > \varepsilon\bigr)
\]
by the proportion of simulated draws for which the deviation from $1/2$ exceeds $\varepsilon$. If $p_i \xrightarrow{\mathbb{P}} 1/2$, then for every fixed $\varepsilon > 0$ these probabilities should decay to zero as $i$ increases. Figure~\ref{fig:str_stability} displays the estimated probabilities for several values of $\varepsilon$. The clear monotone decay provides empirical evidence that $p_i$ converges in probability to $1/2$, consistent with strong stability.

\begin{figure}[H]
\centering
\includegraphics[width=0.8\textwidth]{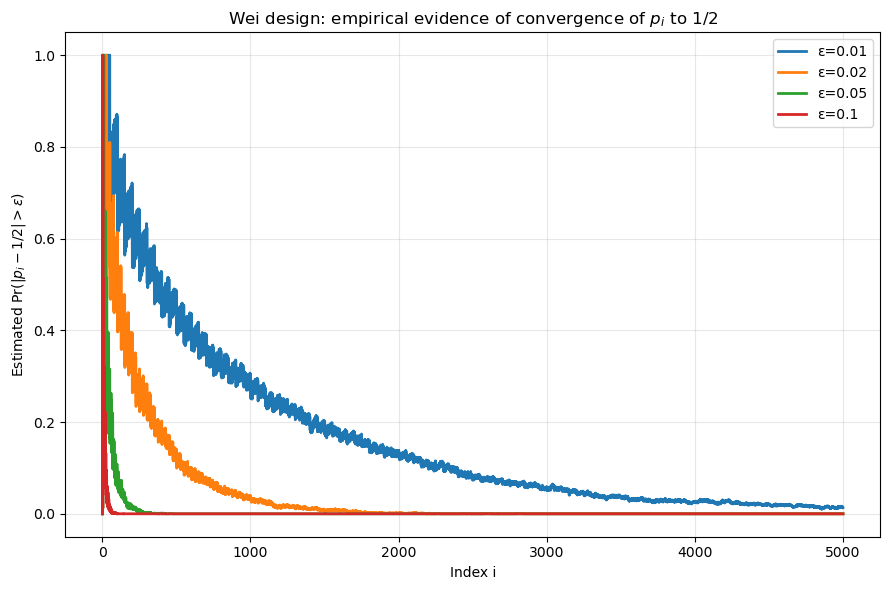}
\caption{Empirical evidence supporting the strong stability of Wei’s design~\cite{Wei1978-gd}.}
\label{fig:str_stability}
\end{figure}


\section{Why weak stability alone does not guarantee condition~\eqref{eqn:extra_assumption}}
\label{sec:weak_stab_alone_doesnt_imply_eqn18}
As discussed in Section~\ref{sec:ipw_results}, the additional condition
\begin{align*}
\frac{1}{N}\sum_{i=1}^N p_i \xrightarrow{\mathbb{P}} \widetilde{p},
\end{align*}
for some constant \(\widetilde{p} \in (0,1)\), is required, in addition to weak design stability, for consistent estimation of \(m_0^2\) and \(m_1^2\), and hence for the conservative variance estimator \(\widehat{\vipwweak}\) for \(\vipwweak\). A natural question is whether weak stability alone suffices to guarantee~\eqref{eqn:extra_assumption}. We show that this is not the case.

Consider two finite blocks of inclusion probabilities, \(A\) and \(B\), such that
\begin{align*}
\frac{1}{|A|} \sum_{p \in A} \frac{1}{p}
&=
\frac{1}{|B|} \sum_{p \in B} \frac{1}{p}, \\
\frac{1}{|A|} \sum_{p \in A} \frac{1}{1-p}
&=
\frac{1}{|B|} \sum_{p \in B} \frac{1}{1-p},
\end{align*}
but
\begin{align*}
\frac{1}{|A|} \sum_{p \in A} p
\neq
\frac{1}{|B|} \sum_{p \in B} p.
\end{align*}
Let the difference between these averages be \(\delta > 0\). Now construct a sequence \(\{p_i\}_{i \ge 1}\) by concatenating copies of the blocks \(A\) and \(B\). By construction, the averages of \(1/p_i\) and \(1/(1-p_i)\) are invariant to the relative frequencies of the blocks, and hence
\begin{align*}
\frac{1}{N} \sum_{i=1}^N \frac{1}{p_i}
\quad \text{and} \quad
\frac{1}{N} \sum_{i=1}^N \frac{1}{1 - p_i}
\end{align*}
converge in probability to fixed constants as \(N \to \infty\). Thus, the resulting design is weakly stable.

In contrast,
\begin{align*}
\frac{1}{N} \sum_{i=1}^N p_i
\end{align*}
depends on the relative proportions of the blocks. In particular, one can construct the sequence recursively as follows: begin with a block of type \(A\), then append copies of block \(B\) until the running average falls below \(1 - \delta/2\); subsequently append copies of block \(A\) until the running average exceeds \(1 - \delta/16\), and continue this procedure iteratively. This yields a sequence for which the running average oscillates between \(1 - \delta/2\) and \(1 - \delta/16\), and hence does not converge, violating condition~\eqref{eqn:extra_assumption}.
 For instance, one may take
\begin{align*}
A = \left\{ \tfrac{1}{11}, \tfrac{1}{6}, \tfrac{1}{6} \right\}, \quad \text{and} \quad
B = \left\{ \tfrac{1}{9}, \tfrac{1}{9}, \tfrac{1}{5} \right\},
\end{align*}
which satisfy the required equalities for the averages of \(1/p\) and \(1/(1-p)\), while having distinct averages of \(p\).

\section{Proofs of Auxiliary Lemmas} \label{sec:proof_auxiliary}
\begin{lemma}\label{lemma: one_lim_pt}
If a sequence $\{x_n\}_{n \geq 1}$ of bounded reals has exactly one limit point $\ell,$ then 
\begin{align*}
\lim_{n \to \infty} x_n = \ell.
\end{align*}
\end{lemma}

\begin{proof}
We argue by contradiction. Suppose $\{x_n\}_{n \geq 1}$ does not converge to $\ell$. Then there exists $\varepsilon > 0$ and a subsequence $\{x_{n_k}\}_{k \geq 1}$ such that  
\begin{align*}
|x_{n_k} - \ell| > \varepsilon \quad \text{for all } k \in \mathbb{N}.
\end{align*}
Since $\{x_{n_k}\}_{k \geq 1}$ is bounded, the Bolzano–Weierstrass theorem ensures the existence of a further subsequence $\{x_{n_{k_j}}\}_{j \geq 1}$ converging to some $\ell'$, implying $\ell'$ is a limit point of $\{x_n\}$. However, since $|x_{n_k} - \ell| > \varepsilon$ for all $k$, we must have 
\begin{align*}
|\ell' - \ell| \ge \varepsilon > 0,
\end{align*}
and hence $\ell' \neq \ell$. This contradicts the assumption that $\ell$ is the unique limit point of $\{x_n\}$. Consequently, we conclude that $x_n \to \ell$ as $n \to \infty$.
\end{proof}

\begin{lemma}\label{lemma:cesaro_inprob}
Let $\{a_i\}_{i \geq 1}$ be a sequence of bounded random variables with $a_i \stackrel{\mathbb{P}}{\rightarrow} a^\star$, then the Cesàro mean converges in probability to the same limit i.e. 
\begin{align*}
\frac{1}{n}\sum_{i=1}^n a_i \stackrel{\mathbb{P}}{\rightarrow} a^\star.   
\end{align*}
\end{lemma}

\begin{proof}
By Markov’s inequality, for any $\varepsilon > 0$,
\begin{align}
\Prob\!\left( \left| \frac{1}{n} \sum_{i=1}^n a_i - a^\star \right| > \varepsilon \right) 
\leq \frac{1}{\varepsilon} \, \E\!\left[ \left| \frac{1}{n} \sum_{i=1}^n a_i - a^\star \right| \right] 
\leq \frac{1}{n \varepsilon} \sum_{i=1}^n \E\!\left[ |a_i - a^\star| \right].
\label{eqn:markov_bound1}
\end{align}
Since $a_i \stackrel{\mathbb{P}}{\rightarrow} a^\star$, we have $|a_i - a^\star| \stackrel{\mathbb{P}}{\rightarrow} 0$. Moreover, boundedness of $\{a_i\}_{i \geq 1}$ implies the existence of an integrable random variable $X$ such that $|a_i| \leq X$ for all $i$. Hence, $|a_i - a^\star| \leq 2X$ and $\E[|a_i - a^\star|] < \infty$ for every $i$. Now, from $a_i \stackrel{\mathbb{P}}{\rightarrow} a^\star$ we may extract a subsequence $a_{i_j} \xrightarrow{a.s.} a^\star$. Dominated convergence theorem then yields
\begin{align*}
\E[|a_{i_j} - a^\star|] \to 0.    
\end{align*}
Thus $0$ is the only possible subsequential limit of $\{\E[|a_i - a^\star|]\}_{i \geq 1}$, hence Lemma~\ref{lemma: one_lim_pt} implies
\begin{align*}
\E[|a_i - a^\star|] \to 0 \quad \text{as } i \to \infty.
\end{align*}
By the Cesàro mean theorem,
\begin{align}
\frac{1}{n}\sum_{i=1}^n \E[|a_i - a^\star|] \to 0.
\label{eqn:cesaro}
\end{align}
Combining~\eqref{eqn:markov_bound1} with \eqref{eqn:cesaro} gives, as $n \to \infty,$
\begin{align*}
\Prob\!\left( \left| \frac{1}{n} \sum_{i=1}^n a_i - a^\star \right| > \varepsilon \right) \to 0,
\end{align*}
and hence the result follows.
\end{proof}

\begin{lemma}\label{lemma:aibi}
Let $\{a_i\}_{i \ge 1}$ be a sequence of bounded random variables with $a_i \stackrel{\mathbb{P}}{\rightarrow} a^\star$, 
and let $\{b_i\}_{i \ge 1}$ be a sequence of bounded real numbers with 
$\frac{1}{n} \sum_{i=1}^n b_i \to b^\star$ as $n \to \infty$.  
Then the cross-average satisfies
\[
\frac{1}{n} \sum_{i=1}^n a_i b_i \stackrel{\mathbb{P}}{\rightarrow} a^\star b^\star.
\]
\end{lemma}

\begin{proof}
Decompose
\begin{align*}
\frac{1}{n}\sum_{i=1}^n a_i b_i
= \underbrace{\frac{1}{n}\sum_{i=1}^n (a_i - a^\star)\, b_i}_{\mathrm{I}}
+ \underbrace{\frac{a^\star}{n}\sum_{i=1}^n b_i}_{\mathrm{II}}.
\end{align*}
Since \(\frac{1}{n}\sum_{i=1}^n b_i \to b^\star\), Slutsky's theorem implies \(\mathrm{II} \stackrel{\mathbb{P}}{\rightarrow} a^\star b^\star\).  
Thus, it suffices to show that \(\mathrm{I} \stackrel{\mathbb{P}}{\rightarrow} 0\).
Since $\{b_i\}_{i \geq 1}$ is bounded, there exists $L>0$ with $|b_i|\le L$ for all $i$. Fix $\varepsilon>0$ and choose $\delta<\tfrac{\varepsilon}{2L}$. Decompose $\mathrm{I}$ as follows:
\begin{align*}
\mathrm{I} = \underbrace{\frac{1}{n}\sum_{i=1}^n (a_i-a^\star)\,b_i\,\mathbf{1}_{\{|a_i-a^\star|>\delta\}}}_{\mathrm{A}}
+ \underbrace{\frac{1}{n}\sum_{i=1}^n (a_i-a^\star)\,b_i\,\mathbf{1}_{\{|a_i-a^\star|\le\delta\}}}_{\mathrm{B}}.
\end{align*}
For $\mathrm{A}$, boundedness of $b_i$ implies
\begin{align*}
|\mathrm{A}| \le \frac{L}{n}\sum_{i=1}^n |a_i-a^\star|.
\end{align*}
Hence, by Markov’s inequality and result~\eqref{eqn:cesaro}, we have, as \(n \to \infty\),
\begin{align*}
\Prob\left(|\mathrm{A}|>\frac{\varepsilon}{2}\right)\le \frac{2L}{n\varepsilon}\sum_{i=1}^n\E|a_i-a^\star|\to 0.
\end{align*}
For \(\mathrm{B}\), we have \(|\mathrm{B}| \le L \delta < \tfrac{\varepsilon}{2}\), 
so that \(\Prob\big(|\mathrm{B}| > \tfrac{\varepsilon}{2}\big) = 0\).  
Hence, \(\Prob(|\mathrm{I}| > \varepsilon) \to 0\), i.e., \(\mathrm{I} \stackrel{\mathbb{P}}{\rightarrow} 0\).  
Combining this with the limit of \(\mathrm{II}\) gives
\begin{align*}
\frac{1}{n} \sum_{i=1}^n a_i b_i \stackrel{\mathbb{P}}{\rightarrow} a^\star b^\star,
\end{align*}
as desired.
\end{proof}

\begin{lemma}\label{lemma:aiibii}
Let $\{a_i\}_{i \ge 1}$ be a sequence of bounded random variables with 
$\frac{1}{n} \sum_{i=1}^n a_i \stackrel{\mathbb{P}}{\rightarrow} a^\star$, and let $\{b_i\}_{i \ge 1}$ be a sequence of bounded real numbers with 
$\frac{1}{n} \sum_{i=1}^n b_i \to b^\star$ as $n \to \infty$.  
Then the cross-average satisfies 
\[
\frac{1}{n} \sum_{i=1}^n a_i b_i \stackrel{\mathbb{P}}{\rightarrow} a^\star b^\star.
\]
\end{lemma}

\begin{proof}
Following the approach in the proof of Lemma~\ref{lemma:aibi}, write
\[
\frac{1}{n}\sum_{i=1}^n a_i b_i
= \underbrace{\frac{1}{n}\sum_{i=1}^n (a_i-a^\star)b_i}_{\mathrm{I}}
+ \underbrace{\frac{1}{n}\sum_{i=1}^n a^\star b_i}_{\mathrm{II}}.
\]
By Slutsky’s theorem,
\[
\mathrm{II} = \frac{a^\star}{n} \sum_{i=1}^n b_i \stackrel{\mathbb{P}}{\rightarrow} a^\star b^\star.
\]
It remains to show that $\mathrm{I} \stackrel{\mathbb{P}}{\rightarrow} 0$. Since $\{b_i\}_{i\ge1}$ is bounded, say $|b_i| \le L$, we have
\[
|\mathrm{I}| \le L \left| \frac{1}{n} \sum_{i=1}^n (a_i - a^\star) \right|,
\]
and therefore, for any $\varepsilon>0$,
\[
\Prob(|\mathrm{I}|>\varepsilon) \le \Prob\left( \left| \frac{1}{n} \sum_{i=1}^n (a_i - a^\star) \right| > \frac{\varepsilon}{L} \right).
\]
Noting that $\frac{1}{n} \sum_{i=1}^n a_i \stackrel{\mathbb{P}}{\rightarrow} a^\star$, the right-hand side converges to $0$ as $n \to \infty$, establishing $\mathrm{I} \stackrel{\mathbb{P}}{\rightarrow} 0$.  
Hence, the claim follows.
\end{proof}

\section{Stationary distribution of $\{S_k\}_{k \geq 1}$}\label{sec:stat_distn}
Let \( P \) be the transition matrix for the discrete-time Markov chain \( \{D_k\}_{k \geq 1} \) and $\boldsymbol{\pi}$ be the corresponding stationary distribution. The $(i,j)$th entry of \( P \), \( p_{i,j} \) denotes the probability of transitioning from state \( i \) to state \( j \) in a single step. Due to the way the setup is defined, $p_{ij} = 0$ for all $j \in \mathbb{N}$ except $j = i-1$ or $j = i+1$. The balance equations from $\boldsymbol{\pi}^\T P = \boldsymbol{\pi}^\T$ are as follows
\begin{align*}
\pi(n) = 
\begin{cases}
\eta\pi(n-1) + (1-\eta)\pi(n+1) & \text{if } n \leq -2, \\
\eta\pi(-2) + \frac{1}{2}\pi(0) & \text{if } n = -1, \\
\eta\pi(-1) + \eta\pi(1) & \text{if } n = 0, \\
\frac{1}{2}\pi(0) + \eta\pi(2) & \text{if } n = 1, \\
(1-\eta)\pi(n-1) + \eta\pi(n+1) & \text{if } n \geq 2.
\end{cases}
\end{align*}

Solving the above set of equations give
\begin{align*}
    \pi(0) = \frac{2\eta-1}{2\eta},\quad \text{and}\quad \pi(-n) = \pi(n) = \frac{2\eta-1}{4\eta(1-\eta)}\left(\frac{1-\eta}{\eta}\right)^n \text{ for } n \in \mathbb{Z} \setminus
\{0\}.
\end{align*}

\end{document}